\documentclass[14pt]{article}
\usepackage{amscd}
\usepackage{amsfonts}
\usepackage{amsmath}
\usepackage{amssymb}
\usepackage{amsthm}
\usepackage{bbm}
\usepackage{CJK}
\usepackage{fancyhdr}
\usepackage{graphicx}
\usepackage{hyperref}
\usepackage{indentfirst}
\usepackage{latexsym}
\usepackage{mathrsfs}
\usepackage{xypic}
\usepackage[all]{xy}           
\usepackage{amsfonts,latexsym}
\usepackage{xspace}
\usepackage{epsfig}
\usepackage{float}
\usepackage{txfonts}
\usepackage{amssymb}
\usepackage{graphicx}
\usepackage{amsmath}
\usepackage{xcolor}
\usepackage[top=1in,bottom=1in,left=1.25in,right=1.25in]{geometry}
\usepackage{epsf,graphicx,epsfig}
\usepackage{amsmath,latexsym,amssymb,amsthm}
\usepackage[nospace,noadjust]{cite}
\usepackage{textcomp}
\usepackage{setspace,cite}
\usepackage{stmaryrd}
\usepackage{tikz-cd} 
\usepackage{tikz}
\usepackage{cancel}
\usetikzlibrary{calc,shapes,arrows,decorations.markings,decorations.pathmorphing}

\usepackage{color}
\usepackage{url}
\usepackage{enumerate}
\usepackage[mathscr]{euscript}

\swapnumbers
    \newtheorem{thm}{Theorem}[section]
    \newtheorem{prop}[thm]{Proposition}
    
    \newtheorem*{Proof*}{Proof}
     
   \newtheorem{lemma}[thm]{Lemma}
   \newtheorem{Note}[thm]{Note}

    \newtheorem{subsec}[thm]{}
\theoremstyle{Definition}
    \newtheorem{Def}[thm]{Definition}
        \newtheorem{Rem}[thm]{Remark}
    \newtheorem{Exam}[thm]{Example}
    
\newtheorem{Coro}[thm]{Corollary}
\theoremstyle{remark}

\usepackage{amssymb}

\usepackage{hyperref}
\hypersetup{
	colorlinks,
	citecolor=blue,
	filecolor=black,
	linkcolor=blue,
	urlcolor=black
}
\tikzset{
  curve/.style={
    settings={#1},
    to path={
      (\tikztostart)
      .. controls ($(\tikztostart)!\pv{pos}!(\tikztotarget)!\pv{height}!270:(\tikztotarget)$)
      and ($(\tikztostart)!1-\pv{pos}!(\tikztotarget)!\pv{height}!270:(\tikztotarget)$)
      .. (\tikztotarget)\tikztonodes
    },
  },
  settings/.code={%
    \tikzset{quiver/.cd,#1}%
    \def\pv##1{\pgfkeysvalueof{/tikz/quiver/##1}}%
  },
  quiver/.cd,
  pos/.initial=0.35,
  height/.initial=0,
}

\date{}
\begin{document}
\renewcommand{\baselinestretch}{1.2}
\renewcommand{\arraystretch}{1.0}
\title{\bf On Hom-Analogues of Heaps and Trusses}
\date{}
\author{{\bf Tarik Anowar$^{1}$,~Ripan Saha$^{2}$\footnote{ Corresponding author:~~Email: ripanjumaths@gmail.com}},~~\bf Sayan Thokdar$^{3}$
        \\
{\small 1.  Department of Mathematics, Raiganj University, Raiganj 733134, West Bengal, India}\\
{\small 2. Department of Mathematics, Raiganj University, Raiganj 733134, West Bengal, India}\\
{\small 3. Department of Mathematics, Indian Institute of Science Education and Research (IISER-Pune)},\\ {\small Pune 411008,  Maharashtra, India.}}

 \maketitle
\begin{center}
\begin{minipage}{12.cm}
\begin{center}{\bf ABSTRACT}\end{center}

This paper introduces Hom-heaps, Hom-trusses, and Hom-braces as Hom-type analogues of their classical counterparts. We establish the correspondence between Hom-heaps and Hom-groups by showing that the retract of a Hom-heap at a point forms a Hom-group precisely when the point is fixed by the twisting map, and prove that translation maps induce isomorphisms between Hom-group retracts at different fixed base points. We introduce three equivalent notions of Hom-trusses and investigate their structural properties. We also propose three variants of Hom-braces and establish their correspondence with Hom-trusses, showing that certain Hom-trusses naturally give rise to Hom-braces and conversely. These results provide a unified framework extending heap and truss theory to the Hom-algebraic setting, with potential applications to the Yang--Baxter equation and non-associative geometry.

\medskip

\medskip

\end{minipage}
\end{center}
\normalsize\vskip0.5cm

\section{Introduction}

Heaps were first studied over a century ago by Prüfer \cite{Pru} and Baer \cite{Bae}, 
though they attracted relatively little attention for many years (see, e.g., \cite{Cer}). 
Defined as algebraic structures equipped with a ternary operation, heaps appear naturally 
in loop theory, symmetric spaces, and non-associative algebra. Every group $(G,\cdot)$ 
yields a heap via $\langle a,b,c\rangle = a \cdot b^{-1} \cdot c$, and conversely, every 
non-empty heap with a distinguished element $e$ recovers a group as its \emph{retract}.
This deep connection with group theory makes heaps a versatile tool for exploring 
generalizations of algebraic structures.

The importance of heaps has resurfaced in recent years through the work of Brzeziński and his collaborators, 
who introduced the notion of \emph{trusses} \cite{Br}. A truss enriches a heap with 
a multiplication compatible with the ternary operation via a distributive law, thus 
providing a unifying framework that encompasses groups, rings, braces, and related systems. 
This perspective has proved useful in the study of set-theoretic solutions of the 
Yang--Baxter equation, as well as in ring theory and category theory. Brzeziński and his collaborators \cite{BBRS} also 
reformulated affine spaces in terms of abelian heaps and developed the affinization of 
algebraic structures such as associative, Lie \cite{BBR}, and Leibniz algebras \cite{BRP}, thereby opening a 
new direction in the study of non-associative geometry.

Braces were first introduced by Wolfgang Rump \cite{Rump} in 2007 as an algebraic structure connecting set-theoretic solutions of the Yang–Baxter equation with radical rings. Since then, braces have become a central tool in the study of non-degenerate involutive set-theoretic solutions \cite{Ced}, providing a unifying framework that links group theory, ring theory, and Hopf algebra techniques. The motivation for introducing braces lies in their ability to translate the combinatorial problem of constructing solutions to the Yang–Baxter equation into an algebraic setting, where the rich interplay between additive and multiplicative structures yields powerful classification and construction methods. This approach not only generalizes earlier results on radical rings but also opens pathways to new applications in quantum groups and related areas.

Hom-algebraic structures arise by twisting classical identities with a distinguished self-map (or “hom” map), so familiar axioms hold only up to that twist. This simple modification produces a far richer class of examples — including q-deformations, endomorphism-twisted algebras and discrete/time-evolution variants — and forces new cohomologies, representation theories and deformation phenomena to appear. Introducing Hom-heaps and Hom-braces lets us carry the powerful heap/brace machinery into this twisted setting: it both generalizes classical classification and extension results and creates avenues to construct new (twisted) set-theoretic solutions of Yang–Baxter–type problems.

\emph{Hom-type algebras} were introduced by Hartwig, Larsson, and Silvestrov \cite{Har} in the context of $q$-deformations. These structures 
generalize classical algebras by twisting defining identities with a self-map~$\alpha$, 
giving rise to notions such as Hom-associativity and Hom-Lie brackets. Since their 
introduction, Hom-algebraic structures have been studied extensively, leading to the 
development of Hom-groups \cite{Has, Lau}, Hom-rings \cite{Ime}, Hom-Lie algebras \cite{Jia,Lau, Ma, Makh, Fre}, and beyond. A Hom-group is a non-associative generalization of a group, where associativity and unitality are twisted by a compatible self-map. Recently, Hom-groups have attracted considerable attention in mathematical research \cite{Has, Hass, Che}.

The aim of this paper is to extend these ideas into the affine and distributive setting 
by introducing and studying \emph{Hom-heaps}, \emph{Hom-trusses}, and \emph{Hom-braces} 
as natural Hom-analogues of their classical counterparts. In the Hom setting, we show 
that the retract of a Hom-heap at a point $x$ forms a Hom-group precisely when $x$ is a 
fixed point of $\alpha$. We also show that from a Hom-group one can define a Hom-heap by twisting the classical approach of the heap setting, thereby generalizing the classical correspondence between heaps and groups. We also investigate translation maps in Hom-heaps and show that they induce isomorphisms between the Hom-group retracts associated with different base points, thereby extending the corresponding classical result for heaps. We also study Hom-subheaps and normal Hom-subheaps, obtaining several interesting results. We further introduce three interrelated definitions of Hom-trusses, motivated by analogous constructions in Hom-ring theory, and establish equivalences between them. For Hom-braces, we present three types of structures and establish connections with Hom-trusses; in 
particular, we prove that every Hom-truss of type $(0)$ in which the multiplicative part 
forms a Hom-group gives rise to a Hom-brace of type $(0)$, and conversely.

The paper therefore lays the foundations for a systematic study of Hom-type structures 
within the affine and distributive framework. By combining heaps, trusses, and Hom-type 
twistings, our work not only generalizes existing algebraic systems but also opens new 
directions for applications, including Yang--Baxter theory, categorical analysis, and 
non-associative geometry.

\section{Preliminaries}

In this section, we recall the fundamental notions of Hom-groups, heaps, and braces that will be used throughout the paper. For further details, we refer the reader to \cite{Bae, Cer, Rum, Rump, BBRS, Che}.

\begin{Def}\cite{Pru,Brz}
A \emph{heap} is a set $H$ equipped with a ternary operation 
\[
\langle -,-,- \rangle : H \times H \times H \longrightarrow H, \quad (a,b,c) \mapsto \langle a,b,c \rangle,
\]
satisfying, for all $a,b,c,d,e \in H$:
\begin{enumerate}
\item $\langle \langle a,b,c \rangle, d, e \rangle = \langle a,b, \langle c,d,e \rangle \rangle$ \quad (associativity),
\item $\langle a,a,b \rangle = b = \langle b,a,a \rangle$ \quad (Mal'cev identity).
\end{enumerate}
\end{Def}

A heap is called \emph{abelian} if $\langle a,b,c\rangle = \langle c,b,a\rangle$ for all $a,b,c \in H$.

\begin{Exam}
Let $(G,\cdot)$ be an abelian group. Then $G$ becomes an abelian heap with the ternary operation 
\[
\langle a,b,c\rangle = a\cdot b^{-1}\cdot c, \quad \text{for all } a,b,c \in G.
\]
\end{Exam}

\begin{Rem}\label{re1}
Every group is a heap under the operation defined in the above example.  
However, the converse is in general not true; for instance, if $H$ is the empty set, then it cannot form a group.  
For any non-empty heap $H$ and an element $e \in H$, the set $H$ endowed with the binary operation  
\[
a +_{e} b = \langle a, e, b \rangle
\]  
is a group (abelian if $H$ is abelian), known as the \emph{retract} of $H$ at $e$, denoted by $G(H,e)$.  
Moreover, the inverse of $a \in H$ in $G(H,e)$ is given by  
\[
a^{-1} = \langle e, a, e \rangle.
\]
\end{Rem}

A heap morphism from $(H,\langle -,-,- \rangle_{1})$ to $(\bar{H},\langle -,-,- \rangle_{2})$ is a map $f:H\to \bar{H}$ such that $f(\langle a,b,c\rangle_{1})=\langle f(a),f(b),f(c)\rangle_{2}$ for all $a,b,c\in H$.

\begin{Def}[\cite{Hass}]
A regular \emph{Hom-group} consists of a set $G$ together with a distinguished element $e \in G$, a bijective map $\alpha : G \to G$, and a binary operation $\bullet : G \times G \to G$, such that the following conditions hold:
\begin{enumerate}
    \item[(i)] For all $a,b,c \in G$,  $\alpha(a) \bullet (b \bullet c) = (a \bullet b) \bullet \alpha(c).$
   
    \item[(ii)] For all $a,b \in G$, $\alpha(a \bullet b) = \alpha(a) \bullet \alpha(b).$
    
    \item[(iii)] The distinguished element $e$ is called the \emph{unit} and satisfies
    
        $$a \bullet e = e \bullet a = \alpha(a),$$
   
    \item[(iv)] For every $a \in G$, there exists an element $a^{-1} \in G$ such that 
    \[
        a \bullet a^{-1} = a^{-1} \bullet a = e.
    \]
\end{enumerate}
\end{Def}

\begin{Rem}
There exists a more general definition of a Hom-group \cite{Lau}, where the map $\alpha$ is not required to be bijective and the inverse condition is also twisted. In this paper, however, we restrict ourselves to the case where $\alpha$ is bijective, and hence we work with \emph{regular Hom-groups}. Throughout the paper, by a Hom-group we shall always mean a regular Hom-group.
\end{Rem}

Let $(G,\bullet)$ be a group and let $\alpha : G \to G$ be a group automorphism. Define a new binary operation 
\[
    a \bullet_{\alpha} b := \alpha(a \bullet b), \quad \text{for all } a,b \in G.
\]
Then the triple $(G, \bullet_{\alpha}, \alpha)$ forms a Hom-group.

\begin{prop}
Let $(G,\bullet,\alpha)$ be a Hom-group and $e$ is an identity element. Then 
\item[(i)]\cite{Jia} $\alpha(e) = e$
\item[(ii)]\cite{Has} For each $a\in G$, $\alpha(a)^{-1}=\alpha(a^{-1})$.
\end{prop}

Let $(G,\bullet,\alpha)$ and $(H,\star,\beta)$ be Hom-groups.  
A map $f: G \rightarrow H$ is called a \emph{homomorphism of Hom-groups} if it satisfies the following conditions: 
\begin{enumerate}
\item[(1)] For all $a,b \in G$, 
$f(a \bullet b) = f(a) \star f(b)$.
\item[(2)] For all $a \in G$, $\beta(f(a)) = f(\alpha(a))$.

\end{enumerate}
Moreover, if $f:G \rightarrow H$ is bijective, then $f$ is called an \emph{isomorphism}.

\begin{Def}\cite{Fre}
Let $R$ be a set together with two binary operations $+:R\times R\rightarrow R$ and $\bullet:R\times R\rightarrow R$, one bijective map $\alpha:R\rightarrow R$ and a special element    $0\in R$ . Then $(R,+,\bullet,0, \alpha)$ is called a Hom-ring if:
\begin{enumerate}
\item[(i)]$(R,+,0)$ is an abelian group.
\item[(ii)] The multiplication is distributive on both sides.
\item[(iii)] $\alpha$ is an abelian group homomorphism.
\item[(iv)] $\alpha$ and $\bullet$ satisfy the Hom-associativity condition; 
$$\alpha(a)\bullet(b\bullet c)=(a\bullet b)\bullet \alpha(c).$$
\end{enumerate}
\end{Def}

A Hom-ring $(R,+,\bullet, 0, \alpha)$ is \emph{unitary} if it has $1\in R$ with  
$a\bullet 1=1\bullet a= a$ for all $a\in R$, and $\alpha(1)=1$.

\begin{Rem}
There are two alternative definitions of a Hom-ring given in \cite{Ime}. In this paper, we introduce the truss versions of these definitions and use them further to establish connections with other algebraic structures.
\end{Rem}

\begin{Def}\cite{Br}
A \emph{truss} is an algebraic system $(T,\langle -,-,- \rangle,\bullet)$ consisting of a ternary operation 
$\langle -,-,- \rangle:T\times T\times T \to T$ such that $(T,\langle -,-,- \rangle)$ is an abelian heap, 
together with an associative binary operation $\bullet:T\times T \to T$ which distributes over 
$\langle -,-,- \rangle$. That is, for all $a,b,c,d \in T$,
\[
a\bullet \langle b,c,d\rangle=\langle a\bullet b,\,a\bullet c,\,a\bullet d\rangle,
\]
and
\[
\langle a,b,c\rangle \bullet d=\langle a\bullet d,\,b\bullet d,\,c\bullet d\rangle.
\]
\end{Def}
\begin{Exam}
Let $H$ be an abelian heap and $E(H)=\{f:H\to H\mid f \text{ is a heap homomorphism}\}$. For $f,g,h\in E(H)$ and $a\in H$, define 
\[
[f,g,h](a)=\langle f(a),g(a),h(a)\rangle.
\] 
With multiplication given by composition, $E(H)$ is a truss.
\end{Exam}

Let $(T_{1},\langle-,-,-\rangle,\bullet)$ and $(T_{2},\langle-,-,-\rangle,\bullet)$ be two trusses.  
A function $\alpha:T_{1}\to T_{2}$ is called a \emph{truss homomorphism} if it is both a heap homomorphism with respect to $\langle-,-,-\rangle$ and a semigroup homomorphism with respect to $\bullet$.

\begin{Def}\cite{Rum}
A set $B$ with binary operations $\star$ and $\bullet$ is called a \emph{brace} if, for all $a,b,c \in B$,
\begin{enumerate}
\item[(i)] $(B,\star)$ is an abelian group;
\item[(ii)] $(B,\bullet)$ is a group;
\item[(iii)] 
\[
a \bullet (b \star c) = (a \bullet b) \star a^{\star} \star (a \bullet c), 
\qquad
(a \star b) \bullet c = (a \bullet c) \star c^{\star} \star (b \bullet c),
\]
\end{enumerate}
where $a^{\star}$ and $c^{\star}$ denote the inverses of $a$ and $c$, respectively, with respect to $\star$.
\end{Def}

A truss $(T,\langle -,-,- \rangle,\bullet)$ is called \emph{unital} if $(T,\bullet)$ is a monoid with identity element $1$.

An element $0$ of a truss $(T,\langle -,-,- \rangle,\bullet)$ is called an \emph{absorber} if, for all $a\in T$,
\[
a\bullet 0 = 0 = 0\bullet a.
\]
Note that an absorber, if it exists, is unique.

\section{Hom-Heap}
In this section, we introduce the notion of a Hom-heap and provide some examples. We also discuss the relation between Hom-groups and Hom-heaps.

\begin{Def}
A \emph{Hom-heap} is a set $H$ together with a bijective map $\alpha:H \to H$ and a ternary operation 
\[
\langle -,-,- \rangle : H \times H \times H \to H,
\]
such that the following axioms are satisfied for all $e,f,x,y,z \in H$:
\begin{enumerate}
\item[(i)] \textbf{Hom-Mal'cev property:}
\[
\langle x,x,y\rangle = \alpha(y) = \langle y,x,x\rangle.
\]
\item[(ii)] \textbf{Hom-associativity property:}
\[
\langle \alpha(e),\alpha(f),\langle x,y,z\rangle \rangle 
= \langle \langle e,f,x\rangle,\alpha(y),\alpha(z)\rangle.
\]
\item[(iii)] \textbf{Compatibility with $\alpha$:}
\[
\alpha\langle x,y,z\rangle = \langle \alpha(x),\alpha(y),\alpha(z)\rangle.
\]
\end{enumerate}
\end{Def}
\begin{Def}
A Hom-heap $(H, \langle -,-,- \rangle, \alpha)$ is called \textbf{involutive} if $\alpha^2 = id$.
\end{Def}

\begin{Exam}
The additive group of integers $(\mathbb{Z}, +)$ admits a Hom-heap structure with  
$\langle -,-,- \rangle : \mathbb{Z} \times \mathbb{Z} \times \mathbb{Z} \to \mathbb{Z}$ and $\alpha : \mathbb{Z} \to \mathbb{Z}$ defined by  
\[
\langle a, b, c \rangle = -(a - b + c), \quad \alpha(a) = -a.
\]
Easy to verify that this is an example of an involutive Hom-heap.
\end{Exam}
\begin{Exam}
Every heap $(H,\langle -,-,- \rangle)$ together with a heap automorphism $\alpha:H\to H$ admits a Hom-heap structure 
\[
(H,\langle -,-,- \rangle_{\alpha},\alpha),
\]
where $\langle x,y,z\rangle_{\alpha} = \alpha(\langle x,y,z\rangle)$.  
In particular, every heap can be regarded as a Hom-heap by taking $\alpha$ to be the identity map.  
Thus, the notion of a Hom-heap generalizes the classical notion of a heap.
\end{Exam}

\begin{Exam}
The structure $(\mathbb{R},+,0,\mathrm{Id})$ is a Hom-group.  
Define the ternary map $\langle -,-,- \rangle : \mathbb{R} \times \mathbb{R} \times \mathbb{R} \to \mathbb{R}$ by  
\[
\langle a,b,c \rangle = a - b + c.
\]  
Then $(\mathbb{R},\langle -,-,- \rangle,\mathrm{Id})$ is a Hom-heap.  
\end{Exam}

\begin{Exam}
Define a ternary operation $\langle -,-,- \rangle$ on $\mathbb{R}$ by  
\[
\langle a,b,c \rangle = \frac{a - b + c}{n}, \quad n \geq 2.
\]  
This operation is not associative in the ternary sense; hence $(\mathbb{R},\langle -,-,- \rangle)$ is not a heap.  
We show that $(\mathbb{R},\langle -,-,- \rangle,\alpha)$ is a Hom-heap, where $\alpha: \mathbb{R} \to \mathbb{R}$ is defined by  
\[
\alpha(a) = \frac{a}{n}.
\]  
It is straightforward to verify that $\alpha$ is a bijection on $\mathbb{R}$. For all $a,b,c \in \mathbb{R}$,  
\[
\alpha\langle a,b,c \rangle = \alpha\left( \frac{a - b + c}{n} \right) = \frac{a - b + c}{n^2},
\]  
\[
\langle \alpha(a),\alpha(b),\alpha(c) \rangle = \left\langle \frac{a}{n},\frac{b}{n},\frac{c}{n} \right\rangle = \frac{a - b + c}{n^2}.
\]  
Thus, $\alpha$ is compatible with $\langle -,-,- \rangle$.  

Moreover, for all $a,b \in \mathbb{R}$,  
\[
\langle a,a,b \rangle = \langle b,a,a \rangle = \frac{a-a+b}{n} = \frac{b}{n} = \alpha(b),
\]  
so the Hom-Mal'cev identity holds.  

Finally, for all $a,b,c,e,f \in \mathbb{R}$,  
\[
\langle \alpha(e),\alpha(f),\langle a,b,c \rangle \rangle 
= \left\langle \frac{e}{n},\frac{f}{n},\frac{a-b+c}{n} \right\rangle 
= \frac{e - f + (a - b + c)}{n^2},
\]  
\[
\langle \langle e,f,a \rangle,\alpha(b),\alpha(c) \rangle 
= \left\langle \frac{e - f + a}{n},\frac{b}{n},\frac{c}{n} \right\rangle 
= \frac{(e - f + a) - b + c}{n^2}.
\]  
Since addition in $\mathbb{R}$ is associative, the Hom-associativity holds.  
Hence $(\mathbb{R},\langle -,-,- \rangle,\alpha)$ is a Hom-heap.
\end{Exam}

\begin{Rem}
The above example defines a Hom-heap, but not a heap, under the given ternary operation.

From Remark~\ref{re1}, we know that for any non-empty heap $H$ and any element $x \in H$, one can construct a binary operation by retracting at $x$, and under this operation $H$ forms a group.  
This naturally raises the question: is it possible to obtain a Hom-group by retracting at an element $x$ in a Hom-heap?

The answer is, in general, \emph{no}. For instance, in the example above we have  
\[
\alpha(x) = \tfrac{x}{n} \neq x.
\]  
However, in a Hom-group the identity element $x$ must satisfy the condition $\alpha(x) = x$.  
Thus, the retract construction that works for heaps does not directly yield a Hom-group in the Hom-setting.

In general, one can obtain a Hom-group from a Hom-heap by retracting at a point $x \in H$ \emph{if and only if} $\alpha(x) = x$.  
In this case, $x$ serves as the identity element in the resulting Hom-group.
\end{Rem}
\begin{prop}\label{Prop1}
Let $(H, \langle -,-,- \rangle, \alpha)$ be a Hom-heap. For all $a,b,c,d,e \in H$, we have
\[
\langle a, b, \langle \alpha^{-1}(c), \alpha^{-1}(d), \alpha^{-1}(e) \rangle \rangle 
= \langle \langle \alpha^{-1}(a), \alpha^{-1}(b), \alpha^{-1}(c) \rangle, d, e \rangle.
\]
\end{prop}

\begin{proof}
Since $\alpha$ is bijective, we have $\alpha(\alpha^{-1}(x)) = x$ for all $x \in H$.  
Using the bijectivity of $\alpha$ and the Hom-associativity property, we obtain
\begin{align*}
&\langle a, b, \langle \alpha^{-1}(c), \alpha^{-1}(d), \alpha^{-1}(e) \rangle \rangle \\
&= \langle \alpha(\alpha^{-1}(a)), \alpha(\alpha^{-1}(b)), \langle \alpha^{-1}(c), \alpha^{-1}(d), \alpha^{-1}(e) \rangle \rangle \\
&= \langle \langle \alpha^{-1}(a), \alpha^{-1}(b), \alpha^{-1}(c) \rangle, d, e \rangle.
\end{align*}
\end{proof}
\begin{thm}\label{hinv}
Let $(H,\langle -,-,-\rangle,\alpha)$ be a Hom-heap and $e\in H$. If $\alpha(e)=e$, then the retract of $H$ at $e$ is a Hom-group $(H,\bullet_{e},\alpha)$ with $a\bullet_{e}b:=\langle a,e,b\rangle$ for all $a,b\in H$.
\end{thm}
\begin{proof}
Let $(H,\langle-,-,-\rangle, \alpha)$ be Hom-heap and $\alpha(e)=e$, for $e\in H$. Since $\alpha$ is bijective, so $\alpha^{-1}(e)=e$. We show that $(H,\bullet_{e},\alpha)$ is Hom-group.
\begin{enumerate}
\item[(1)]For any $a,b,c\in H$, we have
\begin{align*}
\alpha(a)\bullet_{e}(b\bullet_{e}c)&=\alpha(a)\bullet_{e}\langle b,e,c\rangle\\
&=\langle\alpha(a),e,\langle b,e,c\rangle\rangle\\
&=\langle\alpha(a),\alpha(e),\langle b,e,c\rangle\rangle\\
&=\langle\langle a,e,b\rangle,\alpha(e),\alpha(c)\rangle\\
&=\langle(a\bullet_{e}b),e,\alpha(c)\rangle\\
&=(a\bullet_{e}b)\bullet_{e}\alpha(c).
\end{align*}
Therefore, the Hom-associativity hold. 
\item[(2)] For any $a,b\in H$, observe that
\begin{align*}
\alpha(a\bullet_{e}b)&=\alpha\langle a,e,b\rangle\\
&=\langle\alpha(a),\alpha(e),\alpha(b)\rangle\\
&=\langle\alpha(a),e,\alpha(b)\rangle\\
&=\alpha(a)\bullet_{e}\alpha(b).
\end{align*}
Therefore, the map $\alpha$ is multiplicative.
\item[(3)]For any $a\in H$, we get $a\bullet_{e}e=\langle a,e,e\rangle=\alpha(a)$ and $e\bullet_{e}a=\langle e,e,a\rangle=\alpha(a)$. The element $e$ is called unit for $\bullet_{e}$ and it satisfies the Hom-unitary conditions.
\item[(4)]For any $a\in H$, its inverse with respect to the operation $\bullet_{e}$ is given by $a^{-1}=\langle e,\alpha^{-1}(a),e\rangle$, such that 
\begin{align*}
a\bullet_{e}a^{-1}&=\langle a,e,a^{-1}\rangle\\
&=\langle a,e,\langle e,\alpha^{-1}(a),e\rangle\rangle\\
&=\langle a,e,\langle\alpha^{-1}(e),\alpha^{-1}(a),\alpha^{-1}(e)\rangle\rangle\\
&=\langle\langle\alpha^{-1}(a),\alpha^{-1}(e),\alpha^{-1}(e)\rangle,a,e\rangle~~~~\text{(By~Proposition ~~ \ref{Prop1})}\\
&=\langle\alpha(\alpha^{-1}(a)),a,e\rangle\\
&=\alpha(e)\\
&=e.
\end{align*}
Similarly, we can show that $a^{-1}\bullet_{e}a=e$. This proves our result.
\end{enumerate}
\end{proof} 
\begin{thm}
Let $(H, \bullet, 1, \alpha)$ be a Hom-group. Define a ternary operation $\langle -,-,- \rangle : H \times H \times H \to H$ by
\[
\langle a, b, c \rangle = a \bullet(\alpha^{-1}( b^{-1}) \bullet\alpha^{-1} (c)).
\]
Then $(H, \langle -,-,- \rangle, \alpha)$ is a Hom-heap.
\end{thm}

\begin{proof}
Let $a, b, c, d, e \in H$ and let $\alpha : H \to H$ be the given structure map.  

First, we verify that $\alpha$ preserves the ternary operation:
\begin{align*}
\alpha\langle a, b, c\rangle 
&= \alpha(a \bullet(\alpha^{-1}( b^{-1}) \bullet\alpha^{-1}( c))) \\
&= \alpha(a) \bullet(b^{-1} \bullet c) \\
&= \alpha(a) \bullet(\alpha^{-1} \alpha(b^{-1}) \bullet\alpha^{-1} \alpha(c)) \\
&=\alpha(a) \bullet(\alpha^{-1} (\alpha(b)^{-1}) \bullet\alpha^{-1}( \alpha(c)))\\
&=\langle \alpha(a),\alpha(b),\alpha(c)\rangle.
\end{align*}
Hence, $\alpha$ preserves the ternary operation.

Next, we check the Hom-Mal'cev property:
\begin{align*}
\langle a, a, b \rangle &= a \bullet(\alpha^{-1}( a^{-1}) \bullet \alpha
^{-1}(b))\\
&=\alpha\alpha^{-1}(a)\bullet(\alpha^{-1}( a^{-1}) \bullet \alpha
^{-1}(b))\\
&=(\alpha^{-1}(a)\bullet\alpha^{-1}(a^{-1}))\bullet b\\
&=\alpha^{-1}(a\bullet a^{-1})\bullet b\\
&=\alpha^{-1}(1)\bullet b\\
&=1\bullet b\\
&=\alpha(b). 
\end{align*}
and
\begin{align*}
\langle b,a,a\rangle&=b\bullet(\alpha^{-1}(a^{-1})\bullet\alpha^{-1}(a))\\
&=b\bullet \alpha^{-1}(a^{-1}\bullet a)\\
&=b\bullet 1\\
&=\alpha(b).
\end{align*}
Finally, we verify the Hom-heap associativity property:
\begin{align*}
\langle \alpha(a), \alpha(b), \langle c, d, e \rangle \rangle
&= \langle \alpha(a), \alpha(b), (c \bullet( \alpha
^{-1}(d^{-1}) \bullet \alpha^{-1} (e)) \rangle \\
&= \alpha(a) \bullet\{ \alpha^{-1}(\alpha(b)^{-1}) \bullet\alpha^{-1}(c \bullet( \alpha
^{-1}(d^{-1}) \bullet \alpha^{-1} (e)) \}\\
&= \alpha(a) \bullet\{ \alpha(\alpha^{-1}(b^{-1})) \bullet(\alpha^{-1}(c) \bullet( \alpha
^{-2}(d^{-1}) \bullet \alpha^{-2} (e))) \}\\
&= \alpha(a) \bullet\{ (\alpha^{-1}(b^{-1}) \bullet\alpha^{-1}(c)) \bullet( \alpha
^{-1}(d^{-1}) \bullet \alpha^{-1} (e)) \}\\
&=\{a \bullet(\alpha^{-1}(b^{-1}) \bullet\alpha^{-1}(c))\} \bullet( \alpha(\alpha
^{-1}(d^{-1})) \bullet \alpha(\alpha^{-1} (e))) \\
&=\langle a,b,c\rangle\bullet (\alpha^{-1}(\alpha(d^{-1}))\bullet\alpha^{-1}(\alpha(e)))\\
&=\langle\langle a,b,c\rangle,\alpha( d),\alpha( e)\rangle.\\
\end{align*}
Since the expressions agree, $\alpha$ satisfies the Hom-heap associativity property.

Therefore, $(H, \langle -,-,- \rangle, \alpha)$ is a Hom-heap.
\end{proof}
By Theorem 3.10, given a Hom-group $(H, \bullet, 1, \alpha)$, we can construct a Hom-heap $(H, \langle -,-,- \rangle, \alpha)$ defined by
\[
\langle a,b,c\rangle = a \bullet \big(\alpha^{-1}(b^{-1}) \bullet \alpha^{-1}(c)\big),
\]
for all $a,b,c \in H$.
Conversely, by Theorem 3.9, from a Hom-heap one can get a Hom-group structure.
Furthermore, we can easily verify that the map $f : H \to H$, defined by $f(a) = a \bullet x$, for a fixed Hom-retract element $x$ satisfying $\alpha(x)=x \in H$, is an isomorphism of Hom-groups from the initial Hom-group to the corresponding induced Hom-group.\\
By Theorem 3.9, given a Hom-heap $(H,\langle-,-,-\rangle_{H},\alpha)$. We can construct a Hom-group $G(H,\bullet_{e},\alpha)$, where $a\bullet_{e}b=\langle a,e,b\rangle_{H}$, for all $a,b\in H$. Conversely, by Theorem 3.10, every Hom-group induces a Hom-heap $(H,\langle a,b,c\rangle_{G(H)},\alpha)$, where $\langle a,b,c\rangle_{G(H)}=a\bullet_{e}\big{(}(\alpha^{-1}(b))^{-1}\bullet_{e}\alpha^{-1}(c)\big{)}$. Furthermore, using $\alpha(e)=e$, property of Hom-heap and Proposition 3.8 and $a^{-1}=\langle e,\alpha^{-1}(a), e\rangle$, we get
\begin{align*}
\langle a,b,c\rangle_{G(H)}&=a\bullet_{e}\big{(}(\alpha^{-1}(b))^{-1}\bullet_{e}\alpha^{-1}(c)\big{)}\\
&=\langle a,e,\langle(\alpha^{-1}(b))^{-1},e,\alpha^{-1}(c)\rangle_{H}\rangle_{H}\\
&=\langle a,e,\langle\langle e,\alpha^{-2}(b),e\rangle_{H},e,\alpha^{-1}(c)\rangle_{H}\rangle_{H}\\
&=\langle a,e,\langle e,\alpha^{-1}(b),\langle e,e,\alpha^{-2}(c)\rangle_{H}\rangle_{H}\rangle_{H}\\
&=\langle a,e,\langle e,\alpha^{-1}(b),\alpha^{-1}(c)\rangle_{H}\rangle_{H}\\
&=\langle \langle \alpha^{-1}(a),e,e\rangle_{H},b,c\rangle_{H}\\
&=\langle a,b,c\rangle_{H}.
\end{align*}
Therefore, initial and terminal Hom-heap are same.\\
 Let $(H, \langle -, -, - \rangle, \alpha)$ be a Hom-heap. By Theorem 3.9, retracting at different points yields two Hom-groups $G(H, \bullet_o, \alpha)$ and $G(H, \bullet_u, \alpha)$. The translation maps are defined as
\begin{align*}
&T_o^u: G(H, \bullet_o, \alpha) \to G(H, \bullet_u, \alpha), \quad \text{defined by } T_o^u(x) = \langle x, o, u \rangle, \\
&T_u^o: G(H, \bullet_u, \alpha) \to G(H, \bullet_o, \alpha), \quad \text{defined by } T_u^o(y) = \langle \alpha^{-2}(y), u, o \rangle.
\end{align*}
It is straightforward to verify that both maps are Hom-group homomorphisms. Indeed,
\begin{align*}
T_u^o T_o^u (x) &= T_u^o \langle x, o, u \rangle\\
&=\langle \alpha^{-2} (\langle x, o, u \rangle), u, o \rangle \\
&= \langle \langle \alpha^{-2}(x), o, u \rangle, u, o \rangle, \quad \text{since $\alpha$ is a Hom-heap morphism with $\alpha(o) = o$ and $\alpha(u) = u$.} \\
&= \langle \alpha^{-1}(x), o, \langle u, u, o \rangle \rangle, \quad \text{using Proposition 3.8.} \\
&= \langle \alpha^{-1}(x), o, o \rangle \\
&= x.
\end{align*}
A similar computation shows that $T_o^u T_u^o = \mathrm{id}$. Therefore, the two maps are inverses of each other, which implies that the Hom-groups $G(H, \bullet_o, \alpha)$ and $G(H, \bullet_u, \alpha)$ are isomorphic.

\begin{Note}
Every Hom-group is a Hom-heap under the above-mentioned ternary operation.  
However, the converse is in general not true, since the empty set also forms a Hom-heap but not a Hom-group.
\end{Note}
\begin{Def}
Let $(H, \alpha)$ be a Hom-heap with ternary operation $\langle -,-,- \rangle$.  
A subset $S \subseteq H$ is called a \textbf{Hom-subheap} if:
\begin{enumerate}
    \item For all $a,b,c \in S$, we have $\langle a,b,c \rangle \in S$.
    \item  $\alpha(S)= S$.
\end{enumerate}
We denote this by $S \preceq H$.
\end{Def}

\begin{Def}
Let $(H, \langle -,-,- \rangle, \alpha)$ and $(G, \langle -,-,- \rangle, \beta)$ be two Hom-heaps.  
A map $f: H \to G$ is called a \textbf{homomorphism of Hom-heaps} if it satisfies:
\begin{enumerate}
    \item For all $a,b,c \in H$,  
    \[
    f\langle a, b, c \rangle = \langle f(a), f(b), f(c) \rangle.
    \]
    \item For each $a \in H$,  
    \[
    f(\alpha(a)) = \beta(f(a)).
    \]
\end{enumerate}
\end{Def}

\begin{Rem}
If $\alpha$ and $\beta$ are identity maps, then $f$ is a heap homomorphism.
\end{Rem}

\begin{Def}
A Hom-heap homomorphism $f$ is called a \textbf{Hom-heap isomorphism} if it is bijective.
\end{Def}
\begin{Def}
A Hom-subheap $S$ of $(H, \langle -,-,- \rangle, \alpha)$ is called a \textbf{normal Hom-subheap} if there exists $e \in S$ such that, for all $a \in H$ and $s \in S$, there exists $t \in S$ satisfying
\begin{equation}\label{eq12}
\langle a, e, s \rangle = \langle t, e, a \rangle.
\end{equation}
\end{Def}

\begin{lemma}
A Hom-subheap $S$ of $(H, \langle -,-,- \rangle, \alpha)$ is normal if and only if, for all $a \in H$ and $e, s \in S$, there exists $t \in S$ such that
\[
\langle a, e, s \rangle = \langle t, e, a \rangle.
\]
\end{lemma}

\begin{proof}
If such an element $t \in S$ exists for all $a \in H$ and $e, s \in S$ such that equation \eqref{eq12} holds, then by definition $S$ is a normal Hom-subheap of $H$.

Conversely, assume $S$ is a normal Hom-subheap of a Hom-heap $H$. Since $S$ is a Hom-subheap, we have $\alpha(S) \subseteq S$.  
Let $e, f, s \in S$. Then $\alpha^{-1}(e), \alpha^{-1}(f), \alpha^{-1}(s) \in S$. Define
\[
s' = \langle \alpha^{-1}(e), \alpha^{-1}(f), \alpha^{-1}(s) \rangle \in S.
\]
By the definition of a normal Hom-subheap, there exists $t' \in S$ such that, for all $a \in H$,
\begin{align*}
\langle t', e, a \rangle 
&= \langle a, e, s' \rangle \\
&= \langle a, e, \langle \alpha^{-1}(e), \alpha^{-1}(f), \alpha^{-1}(s) \rangle \rangle \quad \text{(substituting $s'$)} \\
&= \langle \langle \alpha^{-1}(a), \alpha^{-1}(e), \alpha^{-1}(e) \rangle, f, s \rangle \quad \text{(by Proposition~\ref{Prop1})} \\
&= \langle \alpha(\alpha^{-1}(a)), f, s \rangle \quad \text{(by the Mal'cev identity for Hom-heaps)} \\
&= \langle a, f, s \rangle.
\end{align*}

Now define
\[
t = \langle \alpha^{-1}(t'), \alpha^{-1}(e), \alpha^{-1}(f) \rangle \in S.
\]
Then
\begin{align*}
\langle t, f, a \rangle 
&= \langle \langle \alpha^{-1}(t'), \alpha^{-1}(e), \alpha^{-1}(f) \rangle, f, s \rangle \\
&= \langle t', e, \langle \alpha^{-1}(f), \alpha^{-1}(f), \alpha^{-1}(a) \rangle \rangle \quad \text{(by Proposition~\ref{Prop1})} \\
&= \langle t', e, \alpha(\alpha^{-1}(a)) \rangle \quad \text{(by the Mal'cev identity for Hom-heaps)} \\
&= \langle t', e, a \rangle.
\end{align*}
Hence, for all $a \in H$ and $f, s \in S$, there exists $t \in S$ such that \eqref{eq12} holds.
\end{proof}

\begin{prop}
Let $S$ be a Hom-subheap of a Hom-heap $H$. Then, for all $a \in H$, the following statements are equivalent:
\begin{enumerate}
    \item For all $e, s \in S$, there exists $t \in S$ such that $\langle a, e, s \rangle = \langle t, e, a \rangle$.
    \item For each $e, s \in S$, $\langle \langle a, e, s \rangle, \alpha(a), \alpha(e) \rangle \in S$.
    \item For each $e \in S$, $\langle \langle a, e, S \rangle, \alpha(a), \alpha(e) \rangle \subseteq S$.
    \item For each $e \in S$, $\langle \langle a, e, S \rangle, \alpha(a), \alpha(e) \rangle = S$.
    \item For each $e \in S$, $\langle \alpha(a), \alpha(e), S \rangle = \langle S, \alpha(e), \alpha(a) \rangle$.
\end{enumerate}
\end{prop}

\begin{proof}
$(1) \Rightarrow (2)$:  
From $(1)$, for all $a \in H$ and $e, s \in S$, there exists $t \in S$ such that $\langle a, e, s \rangle = \langle t, e, a \rangle$. Then
\begin{align*}
\langle \langle a, e, s \rangle, \alpha(a), \alpha(e) \rangle
&= \langle \langle t, e, a \rangle, \alpha(a), \alpha(e) \rangle \\
&= \langle \alpha(t), \alpha(e), \langle a, a, e \rangle \rangle \\
&= \langle \alpha(t), \alpha(e), \alpha(e) \rangle \\
&= \alpha^{2}(t) \in S.
\end{align*}
Hence $(2)$ holds.

\vspace{0.3em}
$(2) \Rightarrow (3)$:  
If $(2)$ holds for all $a \in H$ and all $e, s \in S$, then clearly for each $e \in S$,
\[
\langle \langle a, e, S \rangle, \alpha(a), \alpha(e) \rangle \subseteq S.
\]

\vspace{0.3em}
$(3) \Rightarrow (4)$:  
We want to prove that $\langle \langle a, e, S \rangle, \alpha(a), \alpha(e) \rangle = S$ for all $a \in H$ and $e \in S$.  
From $(3)$, we already have
\[
\langle \langle a, e, S \rangle, \alpha(a), \alpha(e) \rangle \subseteq S.
\]
It remains to prove $S \subseteq \langle \langle a, e, S \rangle, \alpha(a), \alpha(e) \rangle$.

Let $s \in S$. Since $S$ is a Hom-subheap, we have $\alpha(S)=S$. This implies there exists $s' \in S$ such that
\[
s = \alpha(s').
\]
By the Hom-Mal'cev identity,
\[
s = \langle \alpha(s'), \alpha(e), \alpha(e) \rangle
  = \langle \alpha(s'), \alpha(e), \langle a, a, e \rangle \rangle
  = \langle \langle s', e, a \rangle, \alpha(a), \alpha(e) \rangle.
\]
From $(1)$, for all $a \in H$ and $e, s \in S$, there exists $s'' \in S$ such that $\langle a, e, s \rangle = \langle s'', e, a \rangle$. Thus,
\[
s \in \langle \langle a, e, S \rangle, \alpha(a), \alpha(e) \rangle.
\]
Hence $(4)$ holds.

\vspace{0.3em}
$(4) \Rightarrow (5)$:  
Let $\langle s, \alpha(e), \alpha(a) \rangle$ be any element of $\langle S, \alpha(e), \alpha(a) \rangle$.  
Since $S$ is a Hom-subheap, for $s \in S$ there exists $s' \in S$ such that $\alpha^{-1}(s') = s$.  
From $(4)$, there exists $s'' \in S$ such that
\[
s' = \langle \langle a, e, s'' \rangle, \alpha(a), \alpha(e) \rangle.
\]
Then
\begin{align*}
\langle s, \alpha(e), \alpha(a) \rangle
&= \langle \alpha^{-1}(s'), \alpha(e), \alpha(a) \rangle \\
&= \langle \alpha^{-1}(\langle \langle a, e, s'' \rangle, \alpha(a), \alpha(e) \rangle), \alpha(e), \alpha(a) \rangle \\
&= \langle \langle \alpha^{-1} \langle a, e, s'' \rangle, a, e \rangle, \alpha(e), \alpha(a) \rangle \\
&= \langle \alpha(\alpha^{-1} \langle a, e, s'' \rangle), \alpha(a), \langle e, e, a \rangle \rangle \\
&= \langle \langle a, e, s'' \rangle, \alpha(a), \alpha(a) \rangle \\
&= \alpha \langle a, e, s'' \rangle \\
&= \langle \alpha(a),\alpha(e), \alpha(s'') \rangle.
\end{align*}
Let $\alpha(s'') = s''' \in S$. Then $\langle s, \alpha(e), \alpha(a) \rangle \in \langle \alpha(a), \alpha(e), S \rangle$.  
Thus $\langle S, \alpha(e), \alpha(a) \rangle \subseteq \langle \alpha(a), \alpha(e), S \rangle$.  
The reverse inclusion is similar, so $(5)$ holds.

\vspace{0.3em}
$(5) \Rightarrow (1)$:  
For each $a \in H$, there exists $a' \in H$ such that $\alpha(a') = a$.  
From $(5)$, we have
\[
\langle \alpha(a'), \alpha(e), S \rangle = \langle S, \alpha(e), \alpha(a') \rangle.
\]
Putting $\alpha(a') = a$, we get
\[
\langle a, \alpha(e), S \rangle = \langle S, \alpha(e), a \rangle.
\]
That is, for all $a \in H$ and $e, s \in S$, there exists $s' \in S$ such that
\[
\langle a, \alpha(e), s \rangle = \langle s', \alpha(e), a \rangle.
\]
Hence $(1)$ follows.
\end{proof}

\section{Hom-Truss}
In this section, we define Hom-trusses based on abelian heaps in three different but interconnected ways, and also establish some related results.  

\begin{Def}\label{de1}
Let $(T, \langle -,-,- \rangle, \bullet, \alpha)$ be a set equipped with a ternary operation $\langle -,-,- \rangle$, a binary operation $\bullet$, and a bijective unary operation $\alpha : T \rightarrow T$.  
We say that $(T, \langle -,-,- \rangle, \bullet, \alpha)$ is a \textbf{Hom-truss of type $(0)$} if:
\begin{itemize}
    \item[(I)] $(T, \langle -,-,- \rangle)$ is an abelian heap,
    \item[(II)] $\alpha$ is an abelian heap endomorphism, and
    \item[(III)] the following conditions are satisfied:
\end{itemize}
\[
\alpha(a) \bullet (b \bullet c) = (a \bullet b) \bullet \alpha(c) 
\quad \text{(Hom-associativity)},
\]
\[
a \bullet \langle b, c, d \rangle = \langle a \bullet b, a \bullet c, a \bullet d \rangle 
\quad \text{(Left distributive law)},
\]
\[
\langle a, b, c \rangle \bullet d = \langle a \bullet d, b \bullet d, c \bullet d \rangle 
\quad \text{(Right distributive law)}.
\]
\end{Def}
If $(T,\langle-,-,- \rangle,\bullet,\alpha)$ has a unit element $1 \in T$ satisfying  
\begin{equation*}
a \bullet 1 = 1 \bullet a =a,
\end{equation*}
\begin{equation*}
\alpha(1)= 1,
\end{equation*}
for all $a \in T$, then the Hom-truss of type $(0)$, $T$, is said to be \emph{unitary}.
\begin{Exam}
Let $\left( 2\mathbb{Z} + 1, \langle -,-,- \rangle, \bullet \right)$ be a truss, where the ternary operation is given by  
\[
\langle a, b, c \rangle = a - b + c, \quad \text{for all } a,b,c \in 2\mathbb{Z} + 1,
\]  
and the binary operation $\bullet$ is the usual multiplication.  
Consider the map $\alpha : 2\mathbb{Z} + 1 \rightarrow 2\mathbb{Z} + 1$ defined by $\alpha(n) = pn$ for all $n \in 2\mathbb{Z} + 1$, where $p\geq 3$ is prime number.  
We now define a new product
\[
n \bullet_{\alpha} m = \alpha(n \bullet m).
\]  
It is straightforward to verify that  
\[
\left( 2\mathbb{Z} + 1, \langle -,-,- \rangle, \bullet_{\alpha}, \alpha \right)
\]  
is a Hom-truss.
\end{Exam}
\begin{Exam}
Let $(T,\langle -,-,- \rangle,\bullet)$ be a truss and $\alpha:T\rightarrow T$ be a automorphism of truss. Now we define new binary operation $a\bullet_{\alpha}b=\alpha(a\bullet b)$. Then $(T,\langle -,-,- \rangle,\bullet_{\alpha})$ is Hom-truss of type $(0)$.
\end{Exam}
\begin{Def}
A \emph{Hom-truss of type $(1)$} is a tuple 
$(T, \langle -,-,- \rangle, \bullet, \alpha, \beta)$ consisting of a set $T$ equipped with a ternary operation 
\[
\langle -,-,- \rangle : T \times T \times T \to T,
\]
a binary operation 
\[
\bullet : T \times T \to T,
\]
and two bijective unary maps $\alpha, \beta : T \to T$, satisfying the following conditions:
\begin{enumerate}
\item[$(1)$] $(T, \langle -,-,- \rangle, \alpha)$ is an abelian Hom-heap.
\item[$(2)$] $\beta$ is an endomorphism of the abelian Hom-heap $(T, \langle -,-,- \rangle, \alpha)$, i.e.,
\[
\beta\langle a,b,c\rangle = \langle \beta(a), \beta(b), \beta(c) \rangle
\]
for all $a,b,c \in T$, and $\alpha \circ \beta = \beta \circ \alpha$.
\item[$(3)$] $\alpha$ and $\beta$ are both multiplicative maps, i.e.,
\[
\alpha(a \bullet b) = \alpha(a) \bullet \alpha(b) \quad \text{and} \quad \beta(a \bullet b) = \beta(a) \bullet \beta(b),
\]
for all $a,b \in T$.
\item[$(4)$] The map $\beta$ and the binary operation $\bullet$ satisfy the \emph{Hom-associativity} condition:
\begin{equation}\label{eq1}
\beta(a) \bullet (b \bullet c) = (a \bullet b) \bullet \beta(c),
\end{equation}
for all $a,b,c \in T$.
\item[$(5)$] The multiplication is \emph{Hom-distributive} over the ternary operation on both sides:
\begin{equation}\label{eq2}
\alpha(a) \bullet \langle b,c,d \rangle = \langle a \bullet b, a \bullet c, a \bullet d \rangle,
\end{equation}
\begin{equation}\label{eq3}
\langle a,b,c \rangle \bullet \alpha(d) = \langle a \bullet d, b \bullet d, c \bullet d \rangle,
\end{equation}
for all $a,b,c,d \in T$.
\end{enumerate}
\end{Def}

If $(T,\langle-,-,-\rangle,\bullet,\alpha,\beta)$ has a unit element $1 \in T$ satisfying  
\begin{equation}\label{eq4}
a \bullet 1 = 1 \bullet a = \beta(a),
\end{equation}
\begin{equation}\label{eq5}
\alpha(1) = \beta(1) = 1,
\end{equation}
for all $a \in T$, then the Hom-truss of type $(1)$, $T$, is said to be \emph{unitary}.
\begin{Note}\label{Note 4.10}
Note that if $\alpha = \mathrm{id}$, we recover the definition of a Hom-truss of type $(0)$.
\end{Note}
\begin{Rem}
Let $(T,\langle-,-,-\rangle,\bullet,\alpha,\beta)$ be a Hom-truss of type~$(1)$.  
It is called an \emph{$\alpha$-Hom-truss of type~$(1)$} if $\alpha = \beta$.
\end{Rem}

\begin{prop}
Let $(T,\langle-,-,-\rangle,\bullet,\alpha,\beta)$ be a tuple, where $(T,\langle-,-,-\rangle,\alpha)$ is an abelian Hom-heap and $\alpha$ is multiplicative.  
If $(T,\langle-,-,-\rangle,\bullet,\alpha,\beta)$ satisfies \eqref{eq1}, \eqref{eq2} (or \eqref{eq3}), \eqref{eq4}, and \eqref{eq5}, then $\beta$ is an endomorphism of the abelian Hom-heap $(T,\langle-,-,-\rangle,\alpha)$ and is multiplicative.
\end{prop}

\begin{proof}
Since $(T,\langle-,-,-\rangle,\bullet,\alpha,\beta)$ satisfies \eqref{eq1}, \eqref{eq2}, \eqref{eq4}, and \eqref{eq5}, for all $a,b,c \in T$ we have  
\[
\beta\langle a,b,c\rangle
= 1 \bullet \langle a,b,c\rangle
= \alpha(1) \bullet \langle a,b,c\rangle
= \langle 1 \bullet a,\, 1 \bullet b,\, 1 \bullet c\rangle
= \langle \beta(a),\, \beta(b),\, \beta(c)\rangle.
\]
Moreover,  
\[
\alpha \circ \beta(a)
= \alpha(\beta(a))
= \alpha(1 \bullet a)
= \alpha(1) \bullet \alpha(a)
= 1 \bullet \alpha(a)
= \beta(\alpha(a))
= \beta \circ \alpha(a).
\]
The multiplicativity of $\beta$ follows from \eqref{eq1}:  
\[
\beta(a \bullet b)
= (a \bullet b) \bullet 1
= (a \bullet b) \bullet \beta(1)
= \beta(a) \bullet (b \bullet 1)
= \beta(a) \bullet \beta(b).
\]
\end{proof}

\begin{Rem}
As a consequence of the previous proposition, a unitary Hom-truss of type $(1)$ can be defined as a tuple $(T,\langle-,-,-\rangle,\bullet,\alpha,\beta)$ such that $(T,\langle-,-,-\rangle,\alpha)$ is an abelian Hom-heap, $\alpha$ is multiplicative, $\beta$ is bijective, and $(T,\langle-,-,-\rangle,\bullet,\alpha,\beta)$ satisfies the identities (\ref{eq1}), (\ref{eq2}), (\ref{eq3}), (\ref{eq4}), and (\ref{eq5}).
\end{Rem}

Now we introduce another type of Hom-truss as follows:

\begin{Def}
A \emph{Hom-truss of type $(2)$} is a tuple $(T,\langle-,-,-\rangle,\bullet,\alpha,\beta)$ consisting of a set $T$ equipped with a ternary operation $\langle-,-,-\rangle:T\times T\times T\rightarrow T$, a binary operation $\bullet:T\times T\rightarrow T$, and two bijective unary maps $\alpha,\beta:T\rightarrow T$, such that:
\begin{enumerate}
\item[1.] $(T,\langle-,-,-\rangle,\alpha)$ is an abelian Hom-heap.
\item[2.] $\beta$ is an endomorphism of the abelian Hom-heap $(T,\langle-,-,-\rangle,\alpha)$, i.e.,
\[
\beta\langle a,b,c\rangle=\langle\beta(a),\beta(b),\beta(c)\rangle,
\]
for all $a,b,c\in T$, and $\alpha\circ \beta=\beta\circ \alpha$.
\item[3.] Both $\alpha$ and $\beta$ are multiplicative maps, that is,
\[
\alpha(a\bullet b)=\alpha(a)\bullet \alpha(b),\quad \beta(a\bullet b)=\beta(a)\bullet \beta(b),
\]
for all $a,b\in T$.
\item[4.] The maps $\alpha$, $\beta$, and the binary product $\bullet$ satisfy
\begin{equation}\label{eq6}
\beta(\alpha^{2}(a))\bullet\big(\beta(\alpha(b))\bullet \beta^{2}(c)\big)
= \big(\alpha^{2}(a)\bullet \beta(\alpha(b))\big)\bullet \beta^{2}(\alpha(c)),
\end{equation}
for all $a,b,c\in T$.
\item[5.] The multiplication is Hom-distributive over the addition on both sides:
\begin{equation}\label{eq7}
\alpha^{2}(a)\bullet\beta\langle b,c,d\rangle
= \langle \alpha(a)\bullet\beta(b),\alpha(a)\bullet\beta(c),\alpha(a)\bullet\beta(d) \rangle,
\end{equation}
\begin{equation}\label{eq8}
\alpha\langle a,b,c\rangle\bullet\alpha(\beta(d))
= \langle \alpha(a)\bullet\beta(d),\alpha(b)\bullet\beta(d),\alpha(c)\bullet\beta(d) \rangle,
\end{equation}
for all $a,b,c,d\in T$.
\end{enumerate}
\end{Def}

If $(T,\langle-,-,-\rangle,\bullet,\alpha,\beta)$ has a unit element $1\in T$ satisfying
\begin{equation}\label{eq9}
a\bullet 1=\beta(a) \quad\text{and}\quad 1\bullet a=\alpha(a),\quad \text{for all } a\in T,
\end{equation}
\begin{equation}\label{eq10}
\alpha(1)=\beta(1)=1,
\end{equation}
then the Hom-truss of type $(2)$ $T$ is said to be \emph{unitary}.

 \begin{Rem} 
Let $(T, \langle -,-,- \rangle, \bullet, \alpha, \beta)$ be a Hom-truss of type $(2)$.  
It is called an \emph{$\alpha$-Hom-truss of type $(2)$} if $\alpha = \beta$.
\end{Rem} 

\begin{Def}
A Hom-truss $(T, \langle -,-,- \rangle, \bullet, \alpha, \beta)$ is called \emph{commutative} if for all $a, b \in T$ we have  
\[
a \bullet b = b \bullet a.
\]
\end{Def}
\begin{Exam}
Let $(\mathbb{Z}_{6},\langle-,-,-\rangle,\bullet,\alpha,\beta)$ be a Hom-truss of type 2, where $\langle a,b,c\rangle=(a-b+c)$ (mod 6), $a\bullet b=5a+3b$ (mod 6), $\alpha(a)=a$ and $\beta(a)=5a$. This is not a Hom-truss of type $1$.
\end{Exam}

Now we discuss the connections among the above three types of Hom-truss.

\begin{lemma}
Let $(T, \langle -,-,- \rangle, \bullet, \alpha, \beta)$ be a Hom-truss of type $(1)$. Then, for all $a, b, c \in T$:
\begin{enumerate}
\item $\beta(\alpha^{2}(a)) \bullet \big( \beta(\alpha(b)) \bullet \beta^{2}(c) \big) = \alpha(\alpha(a) \bullet \beta(b)) \bullet \beta^{3}(c)$.
\item $\big( \alpha^{2}(a) \bullet \beta(\alpha(b)) \big) \bullet \beta^{2}(\alpha(c)) = (\beta \circ \alpha) \big( \alpha(a) \bullet (b \bullet c) \big)$.
\item $\alpha^{2}(a) \bullet \beta\langle b, c, d \rangle = \langle \alpha(a) \bullet \beta(b), \, \alpha(a) \bullet \beta(c), \, \alpha(a) \bullet \beta(d) \rangle$.
\item $\alpha\langle a, b, c \rangle \bullet \alpha(\beta(d)) = \langle \alpha(a) \bullet \beta(d), \, \alpha(b) \bullet \beta(d), \, \alpha(c) \bullet \beta(d) \rangle$.
\end{enumerate}
\begin{proof}
For $(1)$ and $(2)$, use Property~$(4)$ in the definition of a Hom-truss of type~$(1)$ together with the relation
\[
\alpha\circ\beta=\beta\circ\alpha.
\]
For $(3)$ and $(4)$, use Property~$(5)$.
\end{proof}
\end{lemma}

\begin{lemma}
Let $(T, \langle -,-,- \rangle, \bullet, \alpha, \beta)$ be a Hom-truss of type $(2)$. Then, for all $a, b, c, d \in T$:
\begin{enumerate}
\item $\beta(a) \bullet (b \bullet c) = (a \bullet b) \bullet \alpha(c)$.
\item $(a \bullet b) \bullet \beta(c) = \beta(a) \bullet \big( b \bullet \beta(\alpha^{-1}(c)) \big)$.
\item $\alpha(a) \bullet \langle b, c, d \rangle = \langle a \bullet b, \, a \bullet c, \, a \bullet d \rangle$.
\item $\langle a, b, c \rangle \bullet \alpha(d) = \langle a \bullet d, \, b \bullet d, \, c \bullet d \rangle$.
\end{enumerate}
\end{lemma}
\begin{proof}
For $(1)$ and $(2)$, use Property~$(4)$ in the definition of a Hom-truss of type~$(2)$. For $(3)$ and $(4)$, use Property~$(5)$.
\end{proof}

 From the above two lemmas, we obtain the following simple corollary:
\begin{Coro} \label{Coro 4.16}
An $\alpha$-Hom-truss of type $(1)$ is equivalent to an $\alpha$-Hom-truss of type $(2)$.
\end{Coro}

From Corollary~\ref{Coro 4.16} and Note~\ref{Note 4.10}, we obtain the relationship between different types of Hom-trusses as follows:  
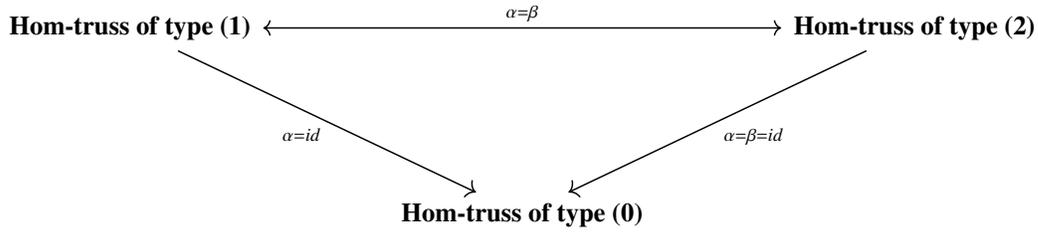
\[
\begin{tikzcd}
	{\textbf{Hom-truss of type (1)}} 
	\arrow[<->, "{\alpha=\beta}"]{rrrr} 
	&&&& {\textbf{Hom-truss of type (2)}} \\
	\\
	\\
	&& {\textbf{Hom-truss of type (0)}}
	\arrow["{\alpha=id}"', from=1-1, to=4-3]
	\arrow["{\alpha=\beta=id}", from=1-5, to=4-3]
\end{tikzcd}
\]
 Now we discuss how, given a truss, one can construct a Hom-truss, and conversely.  

\begin{prop}\label{Prop2}
Let $(T,\langle-,-,-\rangle,\bullet)$ be a truss, and let $\alpha, \beta : T \to T$ be two commuting automorphisms of the truss.  
Define a new ternary operation $\langle-,-,-\rangle_{\ast} : T \times T \times T \to T$ and a new binary multiplication $\bullet_{\ast} : T \times T \to T$ by  
\[
\langle a,b,c\rangle_{\ast} = \alpha\langle a,b,c\rangle = \langle\alpha(a),\alpha(b),\alpha(c)\rangle,
\]
and  
\[
a \bullet_{\ast} b = \beta(a \bullet b) = \beta(a) \bullet \beta(b),
\]
for all $a,b,c \in T$. Then $(T, \langle-,-,-\rangle_{\ast}, \bullet_{\ast}, \alpha, \beta)$ is a Hom-truss of type $(1)$.
\end{prop}
\begin{proof}
For all $a,b,c\in T$, we get $\langle a,a,b\rangle_{\ast}=\langle\alpha(a),\alpha(a),\alpha(b)\rangle=\alpha(b)$. Similarly, we can prove that $\alpha(b)=\langle b,a,a\rangle_{\ast}$. Now for all $a,b,c,d,e\in T$, we get 
\begin{align*}
\langle\alpha(a),\alpha(b),\langle c,d,e\rangle_{\ast}\rangle_{\ast}&=\langle\alpha(a),\alpha(b),\langle\alpha(c),\alpha(d),\alpha(e)\rangle\rangle_{\ast}\\
&=\alpha\langle\alpha(a),\alpha(b),\langle\alpha(c),\alpha(d),\alpha(e)\rangle\rangle\\
&=\alpha\langle\langle\alpha(a),\alpha(b),\alpha(c)\rangle,\alpha(d),\alpha(e)\rangle\\
&=\langle\langle\alpha(a),\alpha(b),\alpha(c)\rangle,\alpha(d),\alpha(e)\rangle_{\ast}\\
&=\langle\langle a,b,c\rangle_{\ast},\alpha(d),\alpha(e)\rangle_{\ast}.
\end{align*}
Also, we get $\alpha\langle a,b,c\rangle_{\ast}=\alpha\circ\alpha\langle a,b,c\rangle=\langle\alpha(a),\alpha(b),\alpha(c)\rangle_{\ast}$. Therefore, $(T,\langle-,-,-\rangle_{\ast},\alpha)$ is an abelian Hom-heap.
Now, we get
\begin{align*}
\beta\langle a,b,c\rangle_{\ast}&=\beta\circ\alpha\langle a,b,c\rangle\\
&=\alpha\circ\beta\langle a,b,c\rangle\\
&=\alpha\langle\beta(a),\beta(b),\beta(c)\rangle\\
&=\langle\beta(a),\beta(b),\beta(c)\rangle_{\ast}.
\end{align*}
We also have
\begin{align*}
&\alpha(a\bullet_{\ast}b)=\alpha\circ\beta(a\bullet b)=\beta\circ\alpha(a\bullet b)=\beta(\alpha(a)\bullet\alpha(b))=\alpha(a)\bullet_{\ast}\alpha(b).\\
&\beta(a\bullet_{\ast}b)=\beta\circ\beta(a\bullet b)=\beta(a)\bullet_{\ast}\beta(b).
\end{align*}
\begin{align*}
\beta(a)\bullet_{\ast}(b\bullet_{\ast}c)&=\beta(a)\bullet_{\ast}\beta(a\bullet b)\\
&=\beta(\beta(a)\bullet(\beta(b)\bullet \beta(c)))\\
&=\beta((\beta(a)\bullet\beta(b))\bullet\beta(c))\\
&=\beta\big((a\bullet_{\ast}b)\bullet\beta(c)\big)\\
&=(a\bullet_{\ast} b)\bullet_{\ast}\beta(c).
\end{align*}
\begin{align*}
\alpha(a)\bullet_{\ast}\langle b,c,d\rangle_{\ast}&=\alpha(a)\bullet_{\ast}\alpha\langle b,c,d\rangle\\
&=\beta(\alpha(a)\bullet\alpha\langle b,c,d\rangle)\\
&=\beta\circ\alpha(a\bullet\langle b,c,d\rangle)\\
&=\alpha\circ\beta(\langle a\bullet b,a\bullet c,a\bullet d\rangle)\\
&=\alpha\langle\beta(a\bullet b),\beta(a\bullet c),\beta(a\bullet d)\rangle\\
&=\alpha\langle a\bullet_{\ast}b,a\bullet_{\ast}c,a\bullet_{\ast}d\rangle\\
&=\langle a\bullet_{\ast}b,a\bullet_{\ast}c,a\bullet_{\ast}d\rangle_{\ast}.
\end{align*}
Similarly, we can prove that $\langle a,b,c\rangle_{\ast}\bullet_{\ast}\alpha(d)=\langle a\bullet_{\ast}d,b\bullet_{\ast}d,c\bullet_{\ast}d\rangle_{\ast}$. Hence, $(T,\langle-,-,-\rangle_{\ast},\bullet_{\ast},\alpha,\beta)$ is a Hom-truss of type $(1)$.
 \end{proof}
 
\begin{prop}
Let $(T,\langle-,-,-\rangle,\bullet,\alpha,\beta)$ be a Hom-truss of type~$(1)$.  
Define a new ternary operation $\langle-,-,-\rangle_{\ast}:T\times T\times T\rightarrow T$ and a new binary operation $\bullet_{\ast}:T\times T\rightarrow T$ by
\[
\langle a,b,c\rangle_{\ast}=\alpha^{-1}\langle a,b,c\rangle
=\langle\alpha^{-1}(a),\alpha^{-1}(b),\alpha^{-1}(c)\rangle,
\]
and 
\[
a\bullet_{\ast}b=\beta^{-1}(a\bullet b)
=\beta^{-1}(a)\bullet \beta^{-1}(b),
\]
for all $a,b,c\in T$.  
Then $(T,\langle-,-,-\rangle_{\ast},\bullet_{\ast})$ is a truss.
\end{prop}
\begin{proof}
It is easy to prove that $(T,\langle-,-,-\rangle_{\ast})$ is an abelian heap and $\bullet_{\ast}$ is a binary operation. Only need to verify that truss distributive property holds. For all a,b,c,d $\in T$, we get
\begin{align*}
a\bullet_{\ast}\langle b,c,d\rangle_{\ast}&=a\bullet_{\ast}\langle\alpha^{-1}(b),\alpha^{-1}(c),\alpha^{-1}(d)\rangle\\
&=\beta^{-1}(a\bullet \langle\alpha^{-1}(b),\alpha^{-1}(c),\alpha^{-1}(d)\rangle)\\
&=\beta^{-1}(\alpha(\alpha^{-1}(a))\bullet \langle\alpha^{-1}(b),\alpha^{-1}(c),\alpha^{-1}(d)\rangle)\\
&=\beta^{-1}\langle \alpha^{-1}(a)\bullet\alpha^{-1}(b),\alpha^{-1}(a)\bullet\alpha^{-1}(c),\alpha^{-1}(a)\bullet\alpha^{-1}(d)\rangle\\
&=\langle\beta^{-1}( \alpha^{-1}(a)\bullet\alpha^{-1}(b)),\beta^{-1}(\alpha^{-1}(a)\bullet\alpha^{-1}(c)),\beta^{-1}(\alpha^{-1}(a)\bullet\alpha^{-1}(d))\rangle\\
&=\langle \alpha^{-1}(a)\bullet_{\ast}\alpha^{-1}(b),\alpha^{-1}(a)\bullet_{\ast}\alpha^{-1}(c),\alpha^{-1}(a)\bullet_{\ast}\alpha^{-1}(d)\rangle\\
&=\langle \alpha^{-1}(a\bullet_{\ast}b),\alpha^{-1}(a\bullet_{\ast}c),\alpha^{-1}(a\bullet_{\ast}d)\rangle\\
&=\langle a\bullet_{\ast}b,a\bullet_{\ast}c,a\bullet_{\ast}d\rangle_\ast.
\end{align*}
Similarly, we can prove that right distributive law.
\end{proof}
\begin{prop}
Let $(T,\langle-,-,-\rangle,\bullet)$ be a truss and let $\alpha,\beta: T\rightarrow T$ be two commuting automorphisms of the truss.  
Define a new ternary operation $\langle-,-,-\rangle_{\ast}:T\times T\times T\rightarrow T$ and a new binary operation $\bullet_{\ast}:T\times T\rightarrow T$ by
\[
\langle a,b,c\rangle_{\ast}=\alpha\langle a,b,c\rangle
=\langle\alpha(a),\alpha(b),\alpha(c)\rangle,
\]
and
\[
a\bullet_{\ast}b=\beta(a)\bullet\beta(b),
\]
for all $a,b\in T$.  
Then $(T,\langle-,-,-\rangle_{\ast},\bullet_{\ast},\alpha,\beta)$ is a Hom-truss of type~$(2)$.
\end{prop}
\begin{proof}
Similar to the above.
\end{proof}
\begin{prop}
Let $(T,\langle-,-,-\rangle,\bullet,\alpha,\beta)$ be a Hom-truss of type $(2)$.  
Define a new ternary operation $\langle-,-,-\rangle_{\ast} : T \times T \times T \to T$  
and a new binary operation $\bullet_{\ast} : T \times T \to T$ by  
\[
\langle a,b,c\rangle_{\ast} = \alpha^{-1}\langle a,b,c\rangle 
= \langle \alpha^{-1}(a),\alpha^{-1}(b),\alpha^{-1}(c) \rangle,
\]
and  
\[
a \bullet_{\ast} b = \beta^{-1}(a) \bullet \alpha^{-1}(b),
\]
for all $a,b \in T$. Then the structure $(T,\langle-,-,-\rangle_{\ast},\bullet_{\ast})$ is a truss.
\end{prop}
\begin{proof}
Similar to the above.
\end{proof}
\begin{Rem}
Here, we first show that starting from a truss, passing to a Hom-truss of type~1, and then applying the inverse construction recovers the original truss.
By Proposition~4.15, every truss
\[
(T,\langle-,-,-\rangle,\bullet)
\]
induces a Hom-truss
\[
(T,\langle-,-,-\rangle_{\ast},\bullet_{\ast},\alpha,\beta),
\]
where
\[
\langle a,b,c\rangle_{\ast}
=\langle\alpha(a),\alpha(b),\alpha(c)\rangle,\qquad
a\bullet_{\ast}b=\beta(a)\bullet\beta(b),
\]
and $\alpha$ and $\beta$ are commuting automorphisms of the truss.

Conversely, by Proposition~4.16, the Hom-truss $(T,\langle-,-,-\rangle_{\ast},\bullet_{\ast},\alpha,\beta)$ induces the truss $(T,\langle-,-,-\rangle_{T},\bullet_{T})$,

where
\[
\langle a,b,c\rangle_{T}
=
\bigl\langle
\alpha^{-1}(a),
\alpha^{-1}(b),
\alpha^{-1}(c)
\bigr\rangle_{\ast},
\]
and
\[
a\bullet_{T} b
=
\beta^{-1}(a)\bullet_{\ast}\beta^{-1}(b).
\]

Now,
\begin{align*}
\langle a,b,c\rangle_{T}
&=
\langle
\alpha^{-1}(a),
\alpha^{-1}(b),
\alpha^{-1}(c)
\rangle_{\ast}\\
&=
\langle
\alpha\alpha^{-1}(a),
\alpha\alpha^{-1}(b),
\alpha\alpha^{-1}(c)
\rangle\\
&=
\langle a,b,c\rangle,
\end{align*}
and
\begin{align*}
a\bullet_{T} b
&=
\beta^{-1}(a)\bullet_{\ast}\beta^{-1}(b)\\
&=
\beta\beta^{-1}(a)\bullet\beta\beta^{-1}(b)\\
&=
a\bullet b.
\end{align*}

Hence, starting from a truss, passing to a Hom-truss of type~1, and then applying the inverse construction recovers the original truss.

Similarly, starting from a truss, passing to a Hom-truss of type~2 (respectively, type~0), and then applying the corresponding inverse construction recovers the original truss.

\end{Rem}
\section{Hom-Brace}
Now we define Hom-braces of different types.

\begin{Def}
A set $B$ together with two binary operations $\star$, $\bullet$ and a bijective map $\alpha:B\rightarrow B$ is called a \emph{Hom-skew left brace of type $(0)$} if, for all $a,b,c\in B$, the following conditions hold:
\begin{enumerate}
\item[(i)] $(B,\star)$ is a group.
\item[(ii)] $(B,\bullet,\alpha)$ is a Hom-group.
\item[(iii)] $\alpha(a\star b)=\alpha(a)\star \alpha(b).$
\item[(iv)] $a\bullet(b\star c)=(a\bullet b)\star a^{\star}\star(a\bullet c),$
\end{enumerate}
where $a^{\star}$ denotes the inverse of $a$ with respect to $\star$.
\end{Def}

If $(B,\star)$ is an abelian group, then a Hom-skew left brace is simply called a \emph{Hom-left brace}. Similarly, a \emph{Hom-skew right brace} is defined using the right distributive condition.

\begin{Def}\label{def6}
A set $B$ together with two binary operations $\star$, $\bullet$ and a bijective map $\alpha:B\rightarrow B$ is called a \emph{Hom-brace of type $(0)$} if, for all $a,b,c\in B$, the following conditions hold:
\begin{enumerate}
\item[(i)] $(B,\star)$ is an abelian group.
\item[(ii)] $(B,\bullet,\alpha)$ is a Hom-group.
\item[(iii)] $\alpha(a\star b)=\alpha(a)\star \alpha(b).$
\item[(iv)] $a\bullet(b\star c)=(a\bullet b)\star a^{\star}\star(a\bullet c),$ \\
 $(b\star c)\bullet a=(b\bullet a)\star a^{\star}\star(c\bullet a),$
\end{enumerate}
where $a^{\star}$ denotes the inverse of $a$ with respect to $\star$.
\end{Def}

Now we define Hom-braces over Hom-groups.  

\begin{Def}
A \emph{Hom-brace of type $(1)$} is a tuple $(B,\star,\bullet,\alpha,\beta)$ consisting of a set $B$ together with two binary operations $\star,\bullet$ and two bijective maps $\alpha,\beta:B\rightarrow B$ such that:
\begin{enumerate}
\item[(i)] $(B,\star,\alpha)$ is an abelian Hom-group.
\item[(ii)] $(B,\bullet,\beta)$ is a Hom-group.
\item[(iii)] $\beta$ is an automorphism of the abelian Hom-group $(B,\star,\alpha)$, i.e.,
$$\beta(a\star b)=\beta(a)\star \beta(b), \quad \forall\, a,b\in B,$$
and $\alpha\circ \beta=\beta\circ\alpha$.
\item[(iv)] $\alpha$ is multiplicative, i.e.,
$$\alpha(a\bullet b)=\alpha(a)\bullet \alpha(b).$$
\item[(v)] The binary operation $\bullet$ is Hom-distributive over $\star$ on both sides:
\begin{align*}
\alpha(a)\bullet(b\star c) &= (a\bullet b)\star \{a^{\star}\star(a\bullet c)\},\\
(b\star c)\bullet\alpha(a) &=\{(b\bullet a)\star a^{\star}\}\star(c\bullet a),
\end{align*}
\end{enumerate}
where $a^{\star}$ denotes the inverse of $a$ with respect to $\star$.
\end{Def}

\begin{Rem}\label{Rem5.4}
Note that if $\alpha=\mathrm{id}$, this definition reduces to that of a Hom-brace of type $(0)$ (see Definition~\ref{def6}).  
If $\alpha=\beta$, then a Hom-brace of type $(1)$ is referred to as an \emph{$\alpha$-Hom-brace of type $(1)$}.
\end{Rem}

Next, we define another type of Hom-brace.  

\begin{Def}
A \emph{Hom-brace of type $(2)$} is a tuple $(B,\star,\bullet,\alpha,\beta)$ consisting of a set $B$ together with two binary operations $\star,\bullet$ and two bijective maps $\alpha,\beta:B\rightarrow B$ such that:
\begin{enumerate}
\item[(i)] $(B,\star,\alpha)$ is an abelian Hom-group.
\item[(ii)] $(B,\bullet,\beta)$ is a Hom-group satisfying the following Hom-associativity condition:
\begin{align*}
\beta(\alpha^{2}(a))\bullet\big(\beta(\alpha(b))\bullet \beta^{2}(c)\big) 
= \big(\alpha^{2}(a)\bullet \beta(\alpha(b))\big)\bullet \beta^{2}(\alpha(c)).
\end{align*}
\item[(iii)] $\beta$ is an automorphism of the abelian Hom-group $(B,\star,\alpha)$, i.e.,
$$\beta(a\star b)=\beta(a)\star \beta(b), \quad \forall\, a,b\in B,$$
and $\alpha\circ \beta=\beta\circ\alpha$.
\item[(iv)] $\alpha$ is multiplicative, i.e.,
$$\alpha(a\bullet b)=\alpha(a)\bullet \alpha(b).$$
\item[(v)] The binary operation $\bullet$ is Hom-distributive over $\star$ on both sides:
\begin{align*}
\alpha^{2}(a)\bullet\beta(b\star c) &= (\alpha(a)\bullet\beta(b))\star\{\alpha(a^{\star})\star(\alpha(a)\bullet\beta(c))\},\\
\alpha(b\star c)\bullet \alpha(\beta(a)) &=\{ (\alpha(b)\bullet\beta(a))\star \alpha(a^{\star})\}\star(\alpha(c)\bullet\beta(a)),
\end{align*}
\end{enumerate}
where $a^{\star}$ denotes the inverse of $a$ with respect to $\star$.
\end{Def}

\begin{Def}
Let $(B,\star,\bullet,\alpha,\beta)$ be a Hom-brace of type $(2)$. It is said to be an \emph{$\alpha$-Hom-brace of type $(2)$} if $\alpha=\beta$.
\end{Def}

\begin{lemma}
Let $(B,\star,\bullet,\alpha,\beta)$ be a Hom-brace of type $(1)$. Then, for all $a,b,c \in B$, where $a^{\star}$ denotes the inverse of $a$ with respect to $\star$, the following identities hold:
\begin{enumerate}
\item $\beta(\alpha^{2}(a)) \bullet \bigl(\beta(\alpha(b)) \bullet \beta^{2}(c)\bigr) 
= \alpha(\alpha(a)\bullet \beta(b)) \bullet \beta^{3}(c)$.
\item $\bigl(\alpha^{2}(a)\bullet \beta(\alpha(b))\bigr)\bullet \beta^{2}(\alpha(c)) 
= (\beta \circ \alpha)\bigl(\alpha(a)\bullet (b\bullet c)\bigr)$.
\item $\alpha^{2}(a)\bullet \beta(b\star c) 
= \bigl(\alpha(a)\bullet \beta(b)\bigr)\star\{\bigl(\alpha(a)\bigr)^{\star} \star \bigl(\alpha(a)\bullet \beta(c)\bigr)\}$.
\item $\alpha(b\star c)\bullet \alpha(\beta(a)) 
=\{ \bigl(\alpha(b)\bullet \beta(a)\bigr)\star (\beta(a))^{\star}\} \star \bigl(\alpha(c)\bullet \beta(a)\bigr)$.
\end{enumerate}
\end{lemma}
\begin{proof}
Statements~$(1)$ and $(2)$ follow from Property~$(4)$ in the definition of a Hom-brace of type~$(1)$ together with the identity
\[
\alpha\circ\beta=\beta\circ\alpha.
\]
Statements~$(3)$ and $(4)$ follow from Property~$(5)$.
\end{proof}

\begin{lemma}
Let $(B,\star,\bullet,\alpha,\beta)$ be a Hom-brace of type $(2)$. Then, for all $a,b,c,d \in B$, where $a^{\star}$ denotes the inverse of $a$ with respect to $\star$, the following hold:
\begin{enumerate}
\item[1.] $\beta(a)\bullet(b\bullet c)=(a\bullet b)\bullet \alpha(c)$.
\item[2.] $(a\bullet b)\bullet\beta(c)=\beta(a)\bullet\bigl(b\bullet\beta(\alpha^{-1}(c))\bigr)$.
\item[3.] $\alpha(a)\bullet(b\star d)=(a\bullet b)\star\{ a^{\star}\star (a\bullet c)\}$.
\item[4.] $(b\star c)\bullet \alpha(a)=\{(b\bullet a)\star a^{\star}\}\star(c\bullet a)$.
\end{enumerate}
\end{lemma}
\begin{proof}
Statements~$(1)$ and $(2)$ follow from Property~$(4)$ in the definition of a Hom-brace of type~$(2)$. Statements~$(3)$ and $(4)$ follow from Property~$(5)$.
\end{proof}

From the above two lemmas, we obtain the following simple corollary:
\begin{Coro}\label{Coro5.9}
An $\alpha$-Hom-brace of type $(1)$ and type $(2)$ are equivalent.
\end{Coro}

From Corollary~\ref{Coro5.9} and Remark~\ref{Rem5.4}, we obtain the relationship between different types of Hom-braces as follows:  

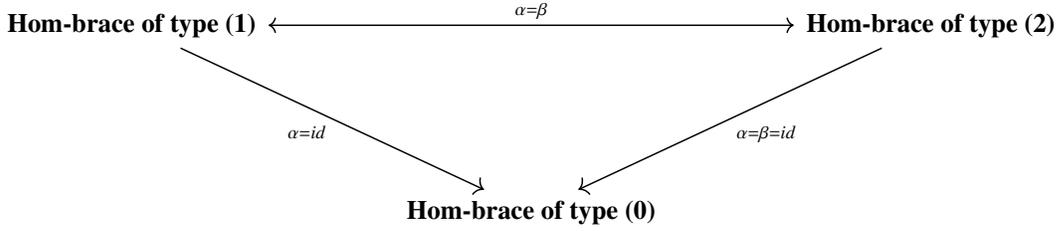
\[
\begin{tikzcd}
	{\textbf{Hom-brace of type (1)}} 
	\arrow[<->, "{\alpha=\beta}"]{rrrr} 
	&&&& {\textbf{Hom-brace of type (2)}} \\
	\\
	\\
	&& {\textbf{Hom-brace of type (0)}}
	\arrow["{\alpha=id}"', from=1-1, to=4-3]
	\arrow["{\alpha=\beta=id}", from=1-5, to=4-3]
\end{tikzcd}
\]

 \section{Relationship Between Hom-Trusses and Hom-Braces}
\begin{Def}
An element $0$ of a Hom-truss is called an \emph{absorber} if, for all $a \in T$,
\[
a \bullet 0 = 0 = 0 \bullet a.
\]
\end{Def}

\begin{thm}
Let $(T,\langle-,-,-\rangle,\bullet,\alpha)$ be a Hom-truss of type $(0)$.
\begin{enumerate}
\item[(1)] If $T$ is unital, then the operations $+_{1}$ and $\bullet$ satisfy the left and right distributive laws Hom-brace of type $(0)$. That is, for all $a,b,c \in T$,
\[
a \bullet (b +_{1} c) = (a \bullet b) -_{1} a +_{1} (a \bullet c),
\]
and
\[
(b +_{1} c) \bullet a = (b \bullet a) -_{1} a +_{1} (c \bullet a).
\]
\item[(2)] If $0$ is an absorber of $T$, then $(T,+_{0},\bullet,\alpha)$ forms a Hom-ring (see the definition of Hom-ring).
\end{enumerate}
\end{thm}

\begin{proof}
To prove assertion $(1)$, we use the heap axioms, the definition of $+_{1}$, the inverse element of the retract at $1$ (as in Remark~\ref{re1}), together with the distributivity and unital properties. We obtain
\begin{align*}
a\bullet(b+_{1}c)&=a\bullet\langle b,1,c\rangle\\
&=\langle a\bullet b,a\bullet 1,a\bullet c\rangle\\
&=\langle a\bullet b,a,a\bullet c\rangle\\
&=\langle\langle a\bullet b,1,1\rangle,a,\langle 1,1,a\bullet c\rangle\rangle\\
&=\langle a\bullet b,1,\langle\langle 1,a,1\rangle,1,a\bullet c\rangle\rangle\\
&=a\bullet b-_{1}a+_{1}a\bullet c.
\end{align*}
The right brace distributive law can be proved in a similar manner. 

To prove assertion $(2)$, let $a,b,c\in T$. Using the definition of $+_{0}$, the distributive law of the Hom-truss~[\ref{de1}], and the property $a\bullet 0=0$, we have
\[
a\bullet(b+_{0}c)=a\bullet\langle b,0,c\rangle=\langle a\bullet b,a\bullet 0,a\bullet c\rangle=\langle a\bullet b,0,a\bullet c\rangle=(a\bullet b)+_{0}(a\bullet c).
\]
Similarly, the right distributive law can be proved.
\end{proof}
\begin{lemma}
Let $(B,+,\bullet,\alpha)$ be a Hom-brace of type $(0)$. Then the following hold:
\begin{enumerate}
\item[(1)] $a\bullet 0 = a$, for all $a \in B$, where $0$ is the identity element with respect to $+$.
\item[(2)] $a\bullet(-c) = 2a - a\bullet c$, for all $a,c \in B$.
\item[(3)] If the ternary operation is defined by 
\[
\langle a,b,c\rangle = a - b + c,
\]
then 
\[
a\bullet \langle b,c,d\rangle = \langle a\bullet b,\; a\bullet c,\; a\bullet d\rangle.
\]
\end{enumerate}
\end{lemma}

\begin{proof}
To prove assertion $(1)$, observe that
\begin{align*}
a\bullet 0 &= a\bullet(0+0) \\
&= a\bullet 0 - a + a\bullet 0 \\
&= 2(a\bullet 0) - a.
\end{align*}
Hence, it follows that $a\bullet 0 = a$.

\smallskip

To prove assertion $(2)$, for any $a,c \in B$, we have
\[
a\bullet 0 = a\bullet(c+(-c)) = a\bullet c - a + a\bullet(-c).
\]
From assertion $(1)$ it follows that
\[
a\bullet (-c) = 2a - a\bullet c.
\]

\smallskip

Finally, to prove assertion $(3)$, for all $a,b,c,d \in B$, we compute:
\begin{align*}
a\bullet \langle b,c,d\rangle &= a\bullet\bigl((b-c)+d\bigr) \\
&= a\bullet(b-c) - a + a\bullet d \\
&= a\bullet b - a + a\bullet(-c) - a + a\bullet d \\
&= a\bullet b - a + (2a - a\bullet c) - a + a\bullet d \quad \text{(using (2))} \\
&= a\bullet b - a\bullet c + a\bullet d \\
&= \langle a\bullet b,\; a\bullet c,\; a\bullet d\rangle.
\end{align*}
This completes the proof.
\end{proof}

\begin{Coro}
Every Hom-truss of type $(0)$, $(T,\langle-,-,-\rangle,\bullet,\alpha)$, in which $(T,\bullet,\alpha)$ is a Hom-group, gives rise to a Hom-brace of type $(0)$, $(T,+_{1},\bullet,\alpha)$. Conversely, any Hom-brace of type $(0)$, $(T,+_{1},\bullet,\alpha)$ gives rise to a unital Hom-truss of type $(0)$.
\end{Coro}

\begin{thm}
Let $(T,\langle-,-,-\rangle,\bullet,\alpha,\beta)$ be an idempotent Hom-truss of type $(1)$.
\begin{enumerate}
\item[(1)] If $T$ is unital, then the operations $+_{1}$ and $\bullet$ satisfy the left and right distributive laws of a Hom-brace of type $(1)$. That is, for all $a,b,c \in T$,
\[
\alpha(a)\bullet(b+_{1}c) = (a\bullet b)+\{-_{1}a+_{1}(a\bullet c)\},
\]
and
\[
(b+_{1}c)\bullet\alpha(a) = \{(b\bullet a)-_{1}a\}+_{1}(c\bullet a).
\]

\item[(2)] If $0$ is an absorber in $T$, then $(T,+_{0},\bullet,\alpha,\beta)$ is a Hom-ring of type $(1)$.
\end{enumerate}
\end{thm}

\begin{proof}
To prove assertion $(1)$, we use the Hom-heap axioms, the definition of $+_{1}$, the inverse element of the retract by $1$ in (\ref{hinv}), the Hom-distributivity of type $(1)$, and the unitary property. We obtain
\begin{align*}
\alpha(a)\bullet(b+_{1}c) &= \alpha(a)\bullet \langle b,1,c\rangle \\
&= \langle a\bullet b,\, a\bullet 1,\, a\bullet c\rangle \\
&= \langle a\bullet b,\, \alpha(a),\, a\bullet c\rangle \\
&= \langle \langle \alpha^{-1}(a\bullet b),1,1\rangle,\, \alpha(a),\, a\bullet c\rangle~~~~~ \text{(using $\langle\alpha^{-1}(a\bullet b),1,1\rangle=\alpha\alpha^{-1}(a\bullet b)=a\bullet b$)} \\
&= \langle \langle \alpha^{-1}(a\bullet b),\, \alpha^{-1}(1),\, \alpha^{-1}(1)\rangle,\, \alpha(a),\, a\bullet c\rangle \\
&= \langle a\bullet b,\, 1,\, \langle 1,\, a,\, \alpha^{-1}(a\bullet c)\rangle\rangle \\
&= \langle a\bullet b,\, 1,\, \langle 1,\, a,\, \langle 1,\, 1,\, \alpha^{-2}(a\bullet c)\rangle\rangle\rangle~~~\text{(using $\langle 1,\, 1,\, \alpha^{-2}(a\bullet c)\rangle=\alpha\alpha^{-2}(a\bullet c)=\alpha^{-1}(a\bullet c)$)} \\
&= \langle a\bullet b,\, 1,\, \langle 1,\, a,\, \langle 1,\, 1,\, \alpha^{-1}(a\bullet c)\rangle\rangle\rangle 
\qquad \text{(since $\alpha^{2}=\alpha$)} \\
&= \langle a\bullet b,\, 1,\, \langle\langle 1,\, \alpha^{-1}(a),\, 1\rangle,\, 1,\, a\bullet c\rangle \rangle
\qquad \text{(by Proposition \ref{Prop1})} \\
&= a\bullet b \,+_{1}\,\{-_{1}a \,+_{1}\, a\bullet c\}.
\end{align*}
Thus, the left brace distributive law holds. The right brace distributive law can be proved in a similar manner. 

\smallskip

To prove assertion $(2)$, take any $a,b,c\in T$. By the definition of $+_{0}$ in \ref{hinv}, the Hom-truss distributivity of type $(1)$, and the property $a\bullet 0=0$, we obtain
\[
 \alpha(a)\bullet(b+_{0}c)
   = \alpha(a)\bullet \langle b,0,c\rangle
   = \langle a\bullet b,\, a\bullet 0,\, a\bullet c\rangle
   = \langle a\bullet b,\, 0,\, a\bullet c\rangle
   = (a\bullet b)+_{0}(a\bullet c).
\]
The right distributive law can be proved analogously.
\end{proof}
\begin{lemma}
Let $(B,+,\bullet,\alpha,\beta)$ be an idempotent Hom-brace of type $(1)$. Then the following hold:
\begin{enumerate}
\item[(1)] $\alpha^{-1}(a)\bullet 0 = \alpha^{-1}(a)$, for all $a \in B$, where $0$ is the identity element with respect to $+$.
\item[(2)] $\alpha^{-1}(a)\bullet(-\alpha^{-1}(c)) = \alpha^{-1}(2a) - \alpha^{-1}(a\bullet c)$, for all $a,c \in B$.
\item[(3)] If the ternary operation is defined by 
\[
\langle a,b,c\rangle = a + \bigl(-\alpha^{-1}(b) + \alpha^{-1}(c)\bigr),
\]
then
\[
\alpha(a)\bullet \langle b,c,d\rangle = \langle a\bullet b,\; a\bullet c,\; a\bullet d\rangle.
\]
\end{enumerate}
\end{lemma}

\begin{proof}
To prove assertion $(1)$, we use the idempotent property and the associativity of the Hom-group. For all $a \in B$, we have
\begin{align*}
\alpha^{-1}(a)\bullet 0 
   &= \alpha^{-1}(a)\bullet (0+0) \\
   &= \alpha(\alpha^{-2}(a))\bullet (0+0) \\
   &= \alpha^{-2}(a)\bullet 0 + \{-\alpha^{-2}(a) + \alpha^{-2}(a)\bullet 0\} \\
   &= \alpha^{-1}(a)\bullet 0 + \{-\alpha^{-1}(a) + \alpha^{-1}(a)\bullet 0\}.
\end{align*}
Hence, it follows that $\alpha^{-1}(a)\bullet 0 = \alpha^{-1}(a)$.

\smallskip

To prove assertion $(2)$, we use assertion $(1)$, the idempotent property, and the Hom-associativity of the Hom-group. For all $a,c \in B$, we obtain
\begin{align*}
\alpha^{-1}(a) 
   &= \alpha^{-1}(a)\bullet 0 \\
   &= \alpha^{-1}(a)\bullet\bigl(\alpha^{-1}(c)+(-\alpha^{-1}(c))\bigr) \\
   &= \alpha(\alpha^{-2}(a))\bullet\bigl(\alpha^{-1}(c)+(-\alpha^{-1}(c))\bigr) \\
   &= \alpha^{-2}(a)\bullet \alpha^{-1}(c) + \{-\alpha^{-2}(a) + (\alpha^{-2}(a)\bullet(-\alpha^{-1}(c)))\} 
      \qquad \text{(by Hom-brace distributivity)} \\
   &= \alpha^{-1}(a)\bullet \alpha^{-1}(c) + \{-\alpha^{-1}(a) + (\alpha^{-1}(a)\bullet(-\alpha^{-1}(c)))\} 
      \qquad \text{(idempotent property)} \\
   &= \alpha^{-1}(a\bullet c) + \{-\alpha^{-1}(a) + (\alpha^{-1}(a)\bullet(-\alpha^{-1}(c)))\} 
      \qquad \text{($\alpha$ is multiplicative)}.
\end{align*}
Thus,
\[
\alpha^{-1}(a)-\alpha^{-1}(a\bullet c) = -\alpha^{-1}(a) + \alpha^{-1}(a)\bullet(-\alpha^{-1}(c)).
\]
Rearranging, we get
\[
\alpha^{-1}(a)\bullet(-\alpha^{-1}(c)) = \alpha^{-1}(2a) - \alpha^{-1}(a\bullet c).
\]

\smallskip

To prove assertion $(3)$, we use assertions $(1)$ and $(2)$ together with Hom-associativity. For all $a,b,c,d \in B$,we have
\begin{align*}
\alpha(a)\bullet\langle b,c,d\rangle 
   &= \alpha(a)\bullet\{\, b+(-\alpha^{-1}(c)+\alpha^{-1}(d)) \,\} \\
   &= a\bullet b + \{-a + [a\bullet(-\alpha^{-1}(c)+\alpha^{-1}(d))]\} \\
   &=a\bullet b+\{-a+[\alpha\alpha^{-1}(a)\bullet(-\alpha^{-1}(c)+\alpha^{-1}(d))] \}\\
   &=a\bullet b+\{-a+[\alpha^{-1}(a)\bullet(-\alpha^{-1}(c))+\{-\alpha^{-1}+\alpha^{-1}(a)\bullet\alpha^{-1}(d)\}]\}\qquad \text{(using Hom-brace distributing property)}\\
   &=a\bullet b+\{-a+[(\alpha^{-1}(2a)-\alpha^{-1}(a\bullet c))+\{-\alpha^{-1}(a)+\alpha^{-1}(a\bullet d)\}]\}\qquad \text{(by assertion (2))}\\
   &=a\bullet b+\{-a+[\alpha(\alpha^{-2}(2a)-\alpha^{-2}(a\bullet c))+\{-\alpha^{-1}(a)+\alpha^{-2}(a\bullet d)\}]\}\qquad \text{(using $\alpha^{-2}=\alpha^{-1})$}\\
   &=a\bullet b+\{-a+[\{(\alpha^{-1}(2a)-\alpha^{-1}(a\bullet c))-\alpha^{-1}(a)\}+\alpha^{-1}(a\bullet d)]\}\qquad\text{(by Hom-associative and $\alpha^{-2}=\alpha^{-1}$)}\\
   &=a\bullet b+\{-a+[\{\alpha^{-1}(a)+(\alpha^{-1}(2a)-\alpha^{-1}(a\bullet c))\}+\alpha^{-1}(a\bullet d)] \}\qquad\text{(using Commutative Hom-group property)}\\
   &=a\bullet b+\{-a+[\{\alpha\alpha^{-2}(-a)+(\alpha^{-1}(2a)-\alpha^{-2}(a\bullet c))\}+\alpha^{-1}(a\bullet d)] \}\\
   &=a\bullet b+\{-a+[\{\alpha^{-1}(-a+2a)-\alpha^{-1}(a\bullet c)\}+\alpha^{-1}(a\bullet d)]\}\qquad \text{(by Hom-associativity)}\\
   &=a\bullet b+\{-\alpha\alpha^{-1}(a)+[\{\alpha^{-1}(a)-\alpha^{-1}(a\bullet c)\}+\alpha^{-2}(a\bullet d)]\}\qquad\text{(using $\alpha^{-2}=\alpha^{-1}$)}\\
   &=a\bullet b+\{[\alpha^{-1}(-a)+\{\alpha^{-1}(a)-\alpha^{-1}(a\bullet c)\}]+\alpha^{-1}(a\bullet d)\}\\
   &=a\bullet b+\{[\alpha\alpha^{-2}(-a)+\{\alpha^{-1}(a)-\alpha^{-2}(a\bullet c)\}]+\alpha^{-1}(a\bullet d)\}\qquad\text{(using $\alpha^{-2}=\alpha^{-1}$)}\\
   &=a\bullet b+\{[\{\alpha^{-1}(-a)+\alpha^{-1}(a)\}-\alpha^{-1}(a\bullet c)]+\alpha^{-1}(a\bullet d)\}\\   
   &= a\bullet b + \{-\alpha^{-1}(a\bullet c) + \alpha^{-1}(a\bullet d)\} \\
   &= \langle a\bullet b,\, a\bullet c,\, a\bullet d\rangle.
\end{align*}
This completes the proof.
\end{proof}

\begin{Coro}
Every idempotent Hom-truss of type $(1)$ $(T,\langle-,-,-\rangle,\bullet,\alpha,\beta)$ in which $(T,\bullet,\beta)$ is a Hom-group gives rise to a Hom-brace of type $(1)$ $(T,+_{1},\bullet,\alpha,\beta)$. Conversely, any idempotent Hom-brace of type $(1)$ $(T,+_{1},\bullet,\alpha,\beta)$ gives rise to a unital Hom-truss of type $(1)$.
\end{Coro}

\begin{thm}
Let $(T,\langle-,-,-\rangle,\bullet,\alpha,\alpha)$ be an idempotent $\alpha$-Hom-truss of type $(2)$.
\begin{enumerate}
\item[(1)] If $T$ is unital, then the operations $+_{1}$ and $\bullet$ satisfy the left and right distributive laws of a Hom-brace of type $(2)$, i.e., for all $a,b,c \in T$,
\[
\alpha^{2}(a)\bullet\alpha(b+_{1}c)=(\alpha(a)\bullet\alpha(b))+\{-_{1}\alpha(a)+_{1}(\alpha(a)\bullet\alpha(c))\},
\]
and
\[
\alpha(b+_{1}c)\bullet\alpha(\alpha(a))=\{(\alpha(b)\bullet\alpha(a))-_{1}\alpha(a)\}+_{1}(\alpha(c)\bullet\alpha(a)).
\]

\item[(2)] If $0$ is an absorber in $T$, then $(T,+_{1},\bullet,\alpha,\alpha)$ is a $\alpha$-Hom-ring of type $(2)$.
\end{enumerate}
\end{thm}

\begin{proof}
To prove assertion $(1)$, we use the Hom-heap axioms, the definition of $+_{1}$, the inverse element of the retract at $1$ in (\ref{hinv}), the Hom-distributivity of type $(2)$, and the unitary property. We obtain
\begin{align*}
 \alpha^{2}(a)\bullet\alpha(b+_{1}c) 
 &= \alpha^{2}(a)\bullet\langle \alpha(b),\alpha(1)=1,\alpha(c)\rangle \\[0.3em]
 &= \langle \alpha(a)\bullet\alpha(b),\,\alpha(a)\bullet 1,\,\alpha(a)\bullet\alpha(c)\rangle \\[0.3em]
 &= \langle \alpha(a)\bullet\alpha(b),\,\alpha^{2}(a),\,\alpha(a)\bullet\alpha(c)\rangle \\[0.3em]
 &= \Big\langle \langle \alpha^{-1}(\alpha(a)\bullet\alpha(b)),\,\alpha^{-1}(1),\,\alpha^{-1}(1)\rangle,\,\alpha^{2}(a),\,\alpha(a)\bullet\alpha(c)\Big\rangle \\[0.3em]
 &= \langle \alpha(a)\bullet\alpha(b),\,1,\,\langle 1,\alpha(a),\,\alpha^{-1}(\alpha(a)\bullet\alpha(c))\rangle \rangle 
      \qquad \text{(by Proposition \ref{Prop1})} \\[0.3em]
 &= \langle \alpha(a)\bullet\alpha(b),\,1,\,\langle 1,\alpha(a),\,\alpha^{-1}(\alpha(a\bullet c))\rangle \rangle \\[0.3em]
 &= \langle \alpha(a)\bullet\alpha(b),\,1,\,\langle 1,\alpha(a),\,a\bullet c \rangle \rangle \\[0.3em]
 &= \langle \alpha(a)\bullet\alpha(b),\,1,\,\langle 1,\alpha(a),\,\langle \alpha^{-1}(1),\alpha^{-1}(1),\alpha^{-1}(\alpha(a)\bullet\alpha(c))\rangle \rangle \rangle \\[0.3em]
 &= \langle \alpha(a)\bullet\alpha(b),\,1,\,\langle \langle 1,\alpha^{-1}(\alpha(a)),1\rangle,1,\,\alpha(a)\bullet\alpha(c)\rangle \rangle \\[0.3em]
 &= (\alpha(a)\bullet\alpha(b)) \;+\;\bigl(-_{1}\alpha(a)+_{1}(\alpha(a)\bullet\alpha(c))\bigr).
\end{align*}
The right Hom-brace distributive law can be proved in a similar manner.

\smallskip

To prove assertion $(2)$, let $a,b,c \in T$. By the definition of $+_{0}$, the type $(2)$ $\alpha$-Hom-truss distributivity, and the property $a\bullet 0=0$, we obtain
\begin{align*}
 \alpha(a)\bullet\alpha(b+_{0}c) 
 &= \alpha(a)\bullet\langle \alpha(b),\alpha(0)=0,\alpha(c)\rangle \\[0.3em]
 &= \langle \alpha(a)\bullet\alpha(b),\,\alpha(a)\bullet 0,\,\alpha(a)\bullet\alpha(c)\rangle \\[0.3em]
 &= \langle \alpha(a)\bullet\alpha(b),\,0,\,\alpha(a)\bullet\alpha(c)\rangle \\[0.3em]
 &= (\alpha(a)\bullet\alpha(b)) +_{0} (\alpha(a)\bullet\alpha(c)).
\end{align*}
Similarly, the right distributive law can be proved.
\end{proof}

\begin{Coro}
Every idempotent $\alpha$-Hom-truss of type $(2)$, 
$(T,\langle -,-,- \rangle,\bullet,\alpha,\alpha)$, 
in which $(T,\bullet,\alpha)$ forms a Hom-group, gives rise to a $\alpha$-Hom-brace of type $(2)$, namely 
$(T,+_{1},\bullet,\alpha,\alpha)$. Conversely, any idempotent $\alpha$-Hom-brace of type $(2)$, 
$(T,+_{1},\bullet,\alpha,\alpha)$, gives rise to a unital $\alpha$-Hom-truss of type $(2)$.
\end{Coro}

\begin{center}
{\bf Future Research Directions}
\end{center}
This paper serves as a foundational contribution to our ongoing research on 
the exploration of \emph{Hom-affgebras}, which may be regarded as the affine 
counterpart of Hom-algebras. In particular, we are in the process of developing 
the theory of \emph{Hom-Lie affgebras}, aiming to extend the notions of 
Lie affgebras into the Hom-setting. Alongside this, our investigation also focuses 
on the representations of \emph{Hom-trusses} as well as \emph{Hom-affgebras}, 
seeking to understand their structural and categorical aspects.

In recent years, Hom-algebraic structures have been studied extensively by many 
mathematicians, motivated by their wide range of applications in deformation theory, 
quantum groups, and non-associative geometry. Inspired by this line of research, 
our goal is to initiate a parallel development of these ideas in the 
\emph{affine world}, making systematic use of the algebraic frameworks provided by 
heaps and trusses. Through this study, we aim not only to build the foundations 
of Hom-affine structures but also to explore the similarities and differences 
that arise when compared with their classical counterparts. This comparative 
approach is expected to shed light on the underlying algebraic phenomena and 
reveal new directions for further research.

\begin{center}
 {\bf ACKNOWLEDGEMENT}
 \end{center}
 
This research is sponsored and supported by the Core Research Grant (CRG) of the Anusandhan National Research Foundation (ANRF), formerly the Science and Engineering Research Board (SERB), under the Department of Science and Technology (DST), Government of India (Grant Number: CRG/2022/005332). All authors gratefully acknowledge the project grant received from the aforementioned agency. Part of this work was carried out when the third author was an ANRF project JRF at Raiganj University. The third author gratefully acknowledges the funding support of ANRF and the facilities provided by Raiganj University.

\section{Declarations}

 {\bf Ethical Approval}

 Not Applicable\\

 {\bf Conflict of interest}
 
 The authors declare that they have no conflict of interest.\\

 {\bf Authors' contributions}
 
 All authors contributed equally. \\

 {\bf Funding}

This research is sponsored and supported by the Core Research Grant (CRG) of the Anusandhan National Research Foundation (ANRF), formerly the Science and Engineering Research Board (SERB), under the Department of Science and Technology (DST), Government of India (Grant Number: CRG/2022/005332).\\

 {\bf Availability of data and materials}
 
 No datasets were used or generated during the preparation of the paper.\\


\begin{thebibliography}{99}

\bibitem{BBR}
R.R. Andruszkiewicz, T. Brzeziński, K. Radziszewski, \emph{Lie affgebras vis-à-vis Lie algebras}, Res. Math. \textbf{80} (2025), art. 61.

\bibitem{Bae}
R. Baer, \emph{Zur Einf\"uhrung des Scharbegriffs}, J. Reine Angew. Math \textbf{160}:199–207, 1929.

\bibitem{Ime}
I. Basdouri, S. Chouaibi, A. Makhlouf, E. Peyghan, Free Hom-groups, Hom-rings and Semisimple modules, Pre-print, -\url{https://arxiv.org/abs/2101.03333}, 2021.

\bibitem{BRP}
T. Brzezi\'nski, K. Radziszewski, B. Ramos P\'erez, \emph{Affinization of algebraic structures: Leibniz algebras}, Journal of Algebra, 693,(2026) 329–361. https://doi.org/10.1016/j.jalgebra.2026.01.018 

\bibitem{BBRS}
S. Breaz, T. Brzezi\'nski, B. Rybolowicz, P. Saracco, \emph{Heaps of modules and affine spaces}, Annal. Mat. Pura Appl. \textbf{203} (2024), 403–445.

\bibitem{Br}
T. Brzezi\'nski, \emph{Trusses: Between braces and rings}, Trans. Amer. Math. Soc. \textbf{372}, 4149–4176 (2019).

\bibitem{Brz}
T. Brzezi\'nski, \emph{Trusses:Paragons, ideals and modules}, Journal of Pure and Applied Algebra
Volume \textbf{224}, Issue 6, June 2020, 106258.

\bibitem{Ced}F. Ced\'o, E. Jespers and J. Okni\'nski, \emph{Braces and the Yang-Baxter equation}, Commun. Math. Phys. \textbf{327} (2014), 101–116.
 \bibitem{Cer}
 J. Certaine, \emph{The ternary operation $(abc) = ab^{-1}c$ of a group}, Bull. Amer. Math. Soc., \textbf{49}:869–877, 1943.
 
 \bibitem{Che}
 L. Chen, T. Feng, Y. Ma, R. Saha, H. Zhang, \emph{On Hom-Groups and Hom-Group actions}, Acta Math. Sin.(Engl. Ser.), \textbf{39}(10)(2023), 1887–1906.
 
\bibitem{Fre}
Y. Fregier and A. Gohr, On Hom-type algebras, \emph{Journal of Generalized Lie Theory and Applications}, Vol. 4 (2010), \emph{Article ID G101001},  \url{doi:10.4303/jglta/G101001}.




\bibitem{Har}J. Hartwig, D. Larsson and S. Silvestrov, \emph{Deformations of Lie algebras using $\sigma$-derivations}, J. Algebra \textbf{295} (2006), 314–36.

\bibitem{Has}
M. Hassanzadeh, \emph{Hom-groups, Representations and homological algebra}, Colloq. Math. \textbf{158} (2019), no. 1, 21–38.

\bibitem{Hass}
M. Hassanzadeh, \emph{Lagrange’s theorem For Hom-Groups}, Rocky Mountain J. Math. \textbf{49} (2019), no. 3, 773–787. 

\bibitem{Jia}
 J. Jiang, S. Kumar Mishra and Y. Sheng, \emph{Hom-Lie algebras and Hom-Lie groups, integration and differentiation}, SIGMA \textbf{16} (2020), 137, 22 pages. 
 
\bibitem{Lau}
C. Laurent-Gengoux, A. Makhlouf, and J. Teles, \emph{Universal algebra of a Hom-Lie algebra and group-like elements}, Journal of Pure and Applied Algebra, Volume \textbf{222}, Issue 5, (2018), P. 1139–1163.

\bibitem{Ma} 
A. Makhlouf and S.D. Silvestrov, \emph{Hom-algebra structures}, J. Gen. Lie Theory Appl. \textbf{2} (2) , 51–64 (2008).

\bibitem{Makh}
A. Makhlouf and S.D. Silvestrov, \emph{Notes on 1-parameter formal deformations of Hom-associative and Hom-Lie algebras}, Forum Math. \textbf{22} (2010), 715–739.

\bibitem{Pru}
H. Pr\"ufer, \emph{Theorie der Abelschen Gruppen. I. Grundeigenschaften}, Math. Z. \textbf{20}:165–187, 1924.

\bibitem{Rum}
W. Rump, \emph{Modules over braces}, Algebra Discrete Math. No. 2 (2006), 127–137.

\bibitem{Rump}
W. Rump, \emph{Braces, radical rings, and the quantum Yang-Baxter equation}, J. Algebra \textbf{307} (2007), 153–170.

\end{thebibliography}
 \end{document}